\documentclass[%
 aip,
 jmp,%
 amsmath,amssymb,
 reprint,%
]{revtex4-1}

\usepackage{graphicx}
\usepackage{dcolumn}
\usepackage{bm}
\usepackage{hyperref}
\usepackage{parcolumns}
\usepackage{subfig}
\usepackage{amsthm}

\usepackage{color,soul} 
\usepackage{ulem}		

\makeatletter

\newcommand{\Rmnum}[1]{\expandafter\@slowromancap\romannumeral #1@}
\makeatother

\hypersetup{
    bookmarks=true,                             
    unicode=false,                              
    pdftoolbar=true,                            
    pdfmenubar=true,                            
    pdffitwindow=true,                          
    pdftitle={Synchronization},      
    pdfauthor={Kramer,Bollt},  
    pdfsubject={Synchronization},                  
    pdfnewwindow=true,                          
    pdfkeywords={keywords},                     
    colorlinks=true,                            
    linkcolor=red,                             
    citecolor=blue,                             
    filecolor=magenta,                          
    urlcolor=blue                               
}

\newtheorem{thm}{Theorem}

\newtheorem{cor}{Corollary}

\newcommand{\eq}[1]{$\displaystyle  #1 $}

\begin{document}

\preprint{AIP/123-QED}

\title[Observer for Occluded Systems]{An Observer for an Occluded Reaction-Diffusion System With Spatially Varying Parameters}

\author{Sean Kramer}
\email{skramer@norwich.edu}
\affiliation{Department of Mathematics, Norwich University}

\author{Erik M. Bollt}
\email{bolltem@clarkson.edu}
\affiliation{Department of Mathematics, Clarkson University}

\date{\today}

  \begin{abstract}
    Spatially dependent parameters of a two-component chaotic reaction-diffusion PDE model describing ocean ecology are observed by sampling a single species. We estimate model parameters and the other species in the system by autosynchronization, where quantities of interest are evolved according to misfit between model and observations, to only partially observed data. Our motivating example comes from oceanic ecology as viewed by remote sensing data, but where noisy occluded data are realized in the form of cloud cover. We demonstrate a method to learn a large-scale coupled synchronizing system that represents spatio-temporal dynamics and apply a network approach to analyze manifold stability. 
  \end{abstract}
  
  
  \maketitle
  
  \begin{quotation} 
  Research in large-scale oceanic phenomena is made possible by remote sensing instruments mounted on ocean-observing satellites. These instruments provide datasets that can be filtered to study sizable ecological events, including harmful algal blooms. The fact that datasets are often patchy when clouds hide regions in the spatial domain is a substantial difficulty when attempting to parameterize a dynamical system. To attack this problem we extend a recently developed autosynchronization method.  Model parameters and states are evolved in a drive-response pattern, on a-priori known model equations, to learn model states and parameters even while data are considerably spatially occluded. It has been shown that, assuming the model structure to be known, a synchronization system can be designed to effectively act as an observer to identify system parameters, even in a large scaled network system.  While a discretized PDE can be interpreted as a particular lattice network, the realistic problem of cloud occlusions will cause times where the observer network is essentially disconnected.  Our prior work has shown that synchronization can exist even in a large scale network that is not fully connected but rather has a so-called fast blinking structure.  The method is analyzed by interpreting the discretized PDE as a large-scale coupled moving neighborhood network.
  \end{quotation}

  \section{Introduction}\label{sec:intro}
  
  Algae form the basis of the food chain in the oceans and are ultimately responsible for providing nourishment for other marine life further up the food chain \cite{fasham93}. Seasonal environmental heterogeneities such as nutrient replenishment, predation, and temperature provide favorable conditions for recurring algal blooms, often called spring blooms. More localized bloom events are commonly observed in estuaries and coastal regions \cite{tb94}. Certain bloom events, especially harmful algal blooms, elicit widespread repercussions on regional communities including human sickness, shellfish poisoning, and fish kills \cite{backer06}. These harmful algal blooms are detrimental to regional ecology and economies through fishing losses and tourism depletion. 
 Models for near-shore algal blooms would be extremely useful for forecasting during such events and might help inform short-term management decisions. 
    
Parameter and state identification based on observed data remains an important topic in both dynamical systems and control theory. Several powerful methods for parameter estimation of spatio-temporal systems include Kalman filter methods \cite{schiff08,annan05,wan00}, multiple shooting methods \cite{Muller04,muller02}, and synchronization methods \cite{PC90,P96,PJK96,SKP96,YCCLP07,YP08,SO09,QBCKA09,SSLP10,BLP11,SP11}. Autosynchronization is a special variation of synchronization methods based on an approach to force a response model to adapt to observed data by developing additional equations for the parameters that depend on the synchronization error \cite{P96,SO09}. Our implementation of the method assumes prior knowledge of the model structure. Recently, it has been shown that it is possible to estimate spatially dependent parameters for a PDE system by autosynchronization using a combination of diffusive and complete replacement coupling of observed data (drive model) to force the response model and parameters to synchronize with observables \cite{kramer2013spatially}. 
  
Our interest here is to exploit these ideas toward modeling ocean ecology as informed by hyperspectral remote sensing data captured by ocean observing satellites. Many well accepted ocean ecology models include predator-prey dynamics between at least two components: zooplankton, the predator, and phytoplankton, the prey \cite{edwards01,upadhyay09,malchow00,M02,freund06,malchow05,truscott94,steffen97,scheffer97,edwards96,truscott94,matthews97,edwards99,fasham90,scheffer91}. Data observations are often noisy or patchy, particularly when observing spatio-temporal systems. The usual hurdle to fitting and subsequently solving a predator-prey reaction-diffusion system as informed by remote sensing data is the inability to observe zooplankton. As of now, there exists no method to estimate zooplankton densities based on hyperspectral inferences. Here, we adapt the method of autosynchronization of PDEs to be used with less available information, where noisy data are occluded by clouds.

At the heart of the problem is the observability of the dynamical system based on available sampling data, in this case phytoplankton. The problem of observability on nonlinear systems has been a topic of research over the past decade and is now much better understood \cite{letellier2005relation,letellier2005graphical,bianco2015symbolic}.  We therefore demonstrate that the system we study is observable from the variable provided by remote sensing data. We note that one might first check if the corresponding ODE system is observable by investigating the invertibility of the Jacobian of the differential embedding map of observed samples \cite{letellier2005relation}. Such a result would provide hope that a search for an autosynchronization scheme is worthwhile. 
  
We begin by introducing the reaction-diffusion equations used to create synthetic observed data. Next, we assign the response system and discuss an autosynchronization configuration. We show the method can work with significant proportions of the data unobservable, e.g. data occluded by cloud cover. Finally, we consider the large-scale coupled synchronization system as a moving neighborhood network and apply a theorem for synchronization based on the rate of switching between network topologies to prove that our system can synchronize. It is shown that as long as the average network corresponding to the graph Laplacian supports synchronization and the switching epoch between new samples of network topologies is small enough, synchronization is achieved. Therefore, it is feasible to realize model fitting and data assimilation for multi-component ecological systems with realistic remote sensing data.


\section{Model Dataset}\label{sec:RD}
Satellite data of plankton blooms often reveal complex mesoscale structures such as ocean gyres and eddies for which there are several theories. As a synthetic dataset, the spatiotemporal model for plankton ecology should have the capability to render mesoscale structures. Medvinksi, et al, \cite{M02} describe a two-component predator-prey model, including phytoplankton and zooplankton, over a rectangular two-dimensional region. Given perturbed initial conditions, the model exhibits spiral patterns on a spatial scale comparable to that which is observed in nature. By sampling snapshots from the solution of this model, we emulate a satellite image dataset. The dataset is complicated by including spatially varying parameters. This is a valid consideration when modeling mesoscale ocean ecology. Consider the system of two PDEs as given in \cite{M02},
\begin{eqnarray} \label{eq:Fish}
\frac{\partial P}{\partial t} &=& \triangle P + P(1-P) - \frac{P Z}{P + h}, \ \ \ \textrm{and}\\
\nonumber \frac{\partial Z}{\partial t} &=& \triangle Z + k\frac{P Z}{P+h} - mZ,
\end{eqnarray} 
where $P(x,y,t)$ represents phytoplankton density, $Z(x,y,t)$ represents zooplankton density, and both are observed on a compact connected two-dimensional domain, $\Omega$, with zero-flux boundary conditions. 

These equations represent a dimensionless reaction-diffusion model for phytoplankton-zooplankton ecology, invoking predator prey dynamics in the reaction term. The ecology is considered over a horizontal layer with homogeneous vertical distributions in the water column. Our simulations are computed over a grid of size $\Omega = 864 \times 288$. The model assumes that phytoplankton obey a logistic growth and are grazed upon by zooplankton following a Holling-type \Rmnum{2} functional response. The Holling-type \Rmnum{2} functional response \cite{Holling59} assumes a decelerating growth rate wherein the predator is limited by its ability to efficiently process food. Zooplankton grow at a rate, $k$, proportional to phytoplankton mortality and die according to a natural mortality rate $m$. For scalar parameters, $k = 2$, $h = 0.4$, and $m = 0.6$, and nonuniform initial conditions, this system gives rise to transient spiral pattern behavior, and progresses into spatially irregular patchy patterns \cite{M02}. We perform numerical simulations with a basic forward-time and central-space discretization using the perturbed initial conditions found in \cite{M02}. 

The system Eq (\ref{eq:Fish}) is modified as found in \cite{M02} by allowing the parameters to be nonnegative $C^0(\Omega)$ functions. Generally, we may allow $\Omega \subset \mathbf{R}^2$ to be a compact domain such as a rectangle for simplicity or a realistic domain representing a coastal region obtained from a satellite. Two examples are found in Figure \ref{fig:figSat}, where high a concentration of phytoplankton appears as a greenish coloring of the water. Imaging sensors mounted on satellites measure light in discrete bandwidths, including several bandwidths outside of the visible range. These bandwidths are subsequently combined to build certain products of interest. To reconstruct an image as the eye would see it, bandwidths in the visible spectrum are combined to build what is called a ``quasi-true'' image. The quasi-true color image at the top of Figure \ref{fig:figSat} was taken on July 8, 2010 from the HICO (Hyperspectral Imager for the Coastal Ocean) instrument mounted on the Japanese Experiment Module Exposed Facility on the International Space Station. It is the first such imaging spectrometer specifically designed to sample the coastal ocean \cite{HICO}. The image captures the Columbia River mouth bordering Oregon and Washington. The domain is large enough to render mesoscale and small scale patterns, which may result from complex intra-species and fluid dynamics. The image at the bottom of Figure \ref{fig:figSat} was taken by the MERIS (Medium Resolution Imaging Spectrometer) instrument on board the Envisat satellite. Here again high phytoplankton concentrations appear as a greenish coloring in the water. This image highlights a presently unavoidable issue with hyperspectral satellite data: the presence of cloud coverage. 
\begin{figure}[!h]
\begin{center} 
\begin{minipage}{.5\textwidth}
\hspace{-1cm}
{\includegraphics[width=3.2in]{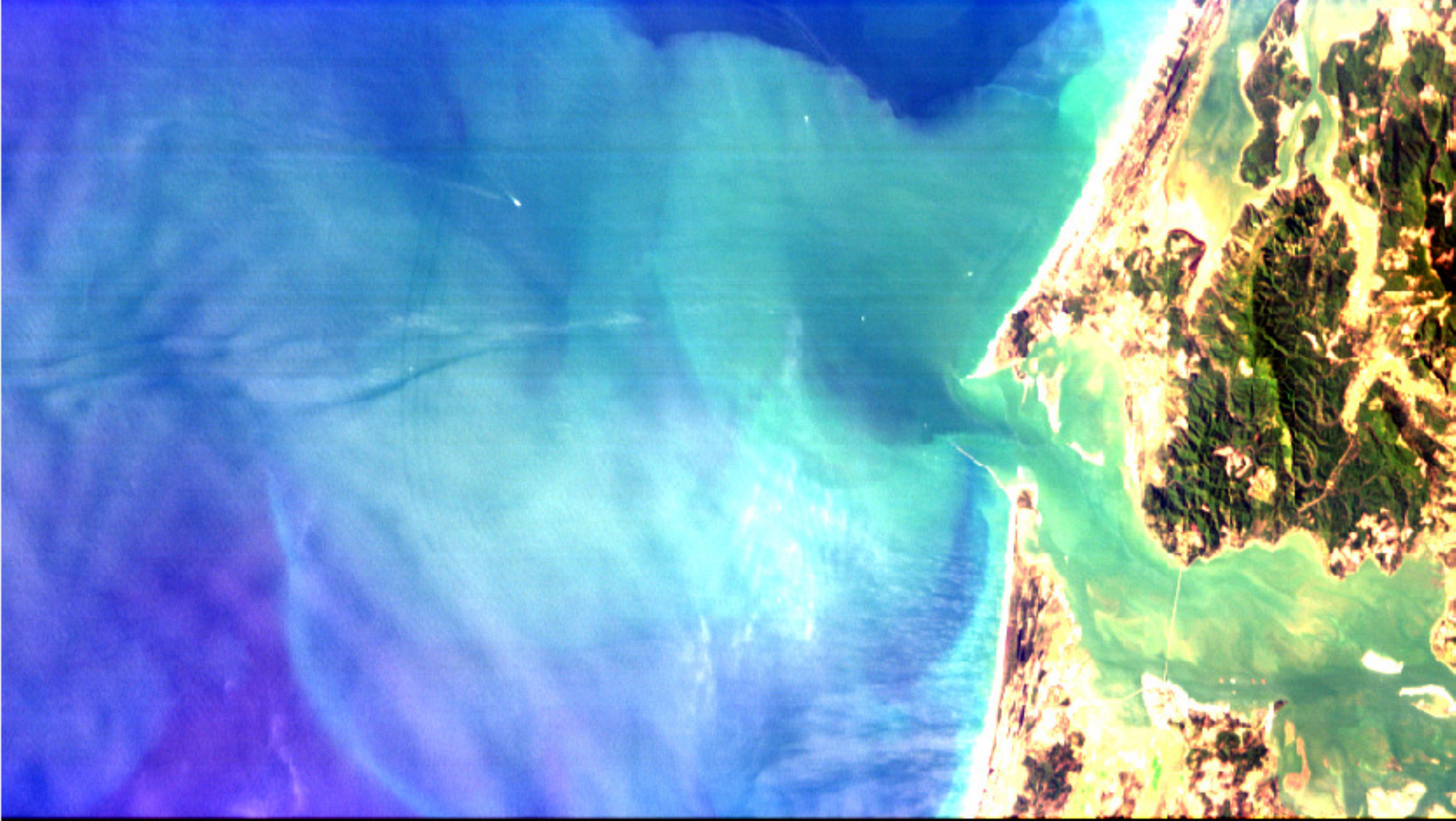}}
\end{minipage}
\begin{minipage}{.5\textwidth}
\hspace{1.5cm}
{\includegraphics[width = 1.9in]{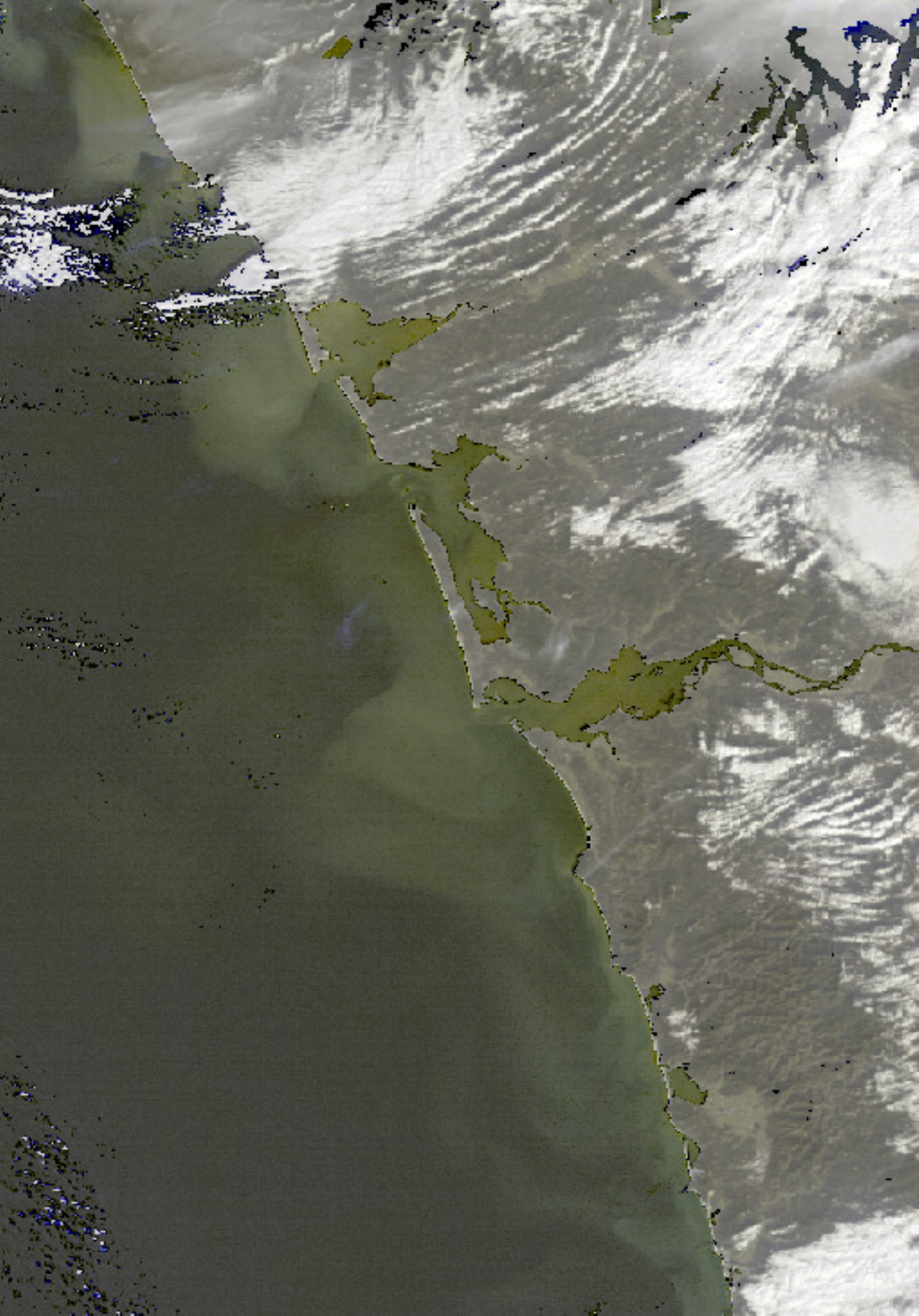}}
\end{minipage}
\caption{A quasi-true color satellite image from HICO instrument, \cite{HICO} (top), of the Columbia River mouth taken on July 8, 2010. High plankton densities shown by green coloring of the water. Spatial resolution is fine enough that a boat is clearly visible in the upper half of the image. Bottom: Quasi-true color image of same region taken during an algal bloom on December 12, 2009 by the MERIS instrument on the ENVISAT satellite.}
\label{fig:figSat}
\end{center}
\end{figure}

In many systems, it is quite reasonable to expect that model parameters need not be spatially homogeneous. And therefore, taking our problem of interest, spatial inhomogeneity in parameter values may be an important theoretical assumption when constructing models for coastal algal blooms, since the plankton growth rate is affected by near-shore nutrient runoff and upwelling \cite{thomann1975mathematical,M02,mattern2012estimating}. More to that point, ocean fronts and eddies cause flow-induced long-term inhomogeneities in the ocean which results in a formidable spatial structure for density profiles in the ocean \cite{M02}. Whether inhomogeneities be the result of the flow dynamics or of boundary conditions from nutrient runoff, they are an important consideration for modelling ecology over large coastal domains. Thus, depending on the scale and resolution, it may be prudent to include spatially dependent parameters.

Therefore, we develop synthetic datasets with spatially varying parameters to challenge our methods. To push our methods we add random noise to each parameter as displayed in Figure \ref{fig:Parameters}. Spatially dependent parameters are chosen to be in the range given in \cite{M02} for spatially irregular behavior. Three different functional forms for the parameters are tested for variety. First, we define a Gaussian parameter function,

\begin{eqnarray}\label{eq:Gauss}
k_1(x,y) &=& a e^{-\left(\frac{(x-n/2)^2}{2\sigma^2} + \frac{(y-m/2)^2}{2\sigma^2}\right)}, \ \ \ \textrm{and} \\
\nonumber m_1(x,y) &=& c e^{-\left(\frac{(x-n/2)^2}{2\sigma^2} + \frac{(y-m/2)^2}{2\sigma^2}\right)},
\end{eqnarray}
where $a = 2, c = 0.6, m = 300, n = 900$, and $\sigma = 400$.  Appropriate parameters are chosen to maintain $m(x,y)$ and $k(x,y)$ in the target range. Figure \ref{fig:Parameters} shows the three parameter forms discussed above, where only $k(x,y)$ is plotted since the parameters differ by a scalar multiple. For example, Eq (\ref{eq:Gauss}) is displayed in Figure \ref{fig:k_gauss}. Next, we define,

\begin{eqnarray}\label{eq:Sineplot}
k_2(x,y) &=& a \ \cos(b x + d)\sin(by) + s, \ \ \ \textrm{and}\\
\nonumber m_2(x,y) &=& c \ \cos(b x + d)\sin(by) + t,
\end{eqnarray}
where $a = 0.2, b = \pi/(m/2), c = 0.6, d = \pi/2, s = 0.5$, and $t = 1.5$, to test the quality of the autosynchronization method to resolve fine spatial structures in model parameters. The surfaces produced by Eq (\ref{eq:Sineplot}) are displayed in Figure \ref{fig:k_sin}. 

Finally, we build a swirly parameter function in order to simulate spiral-like behavior in parameter values as might be expected in turbulent coastal regions. A time instance is sampled from a simulation of the original PDE, Eq (\ref{eq:Fish}), is scaled appropriately, and is treated as a parameter function. These spiral parameters, $k_3(x,y)$, are shown in Figure \ref{fig:k_swirl}. 

\begin{figure}[!h]
\centering \label{fig:Parameters}
\subfloat[][]{\includegraphics[width=.45\textwidth]{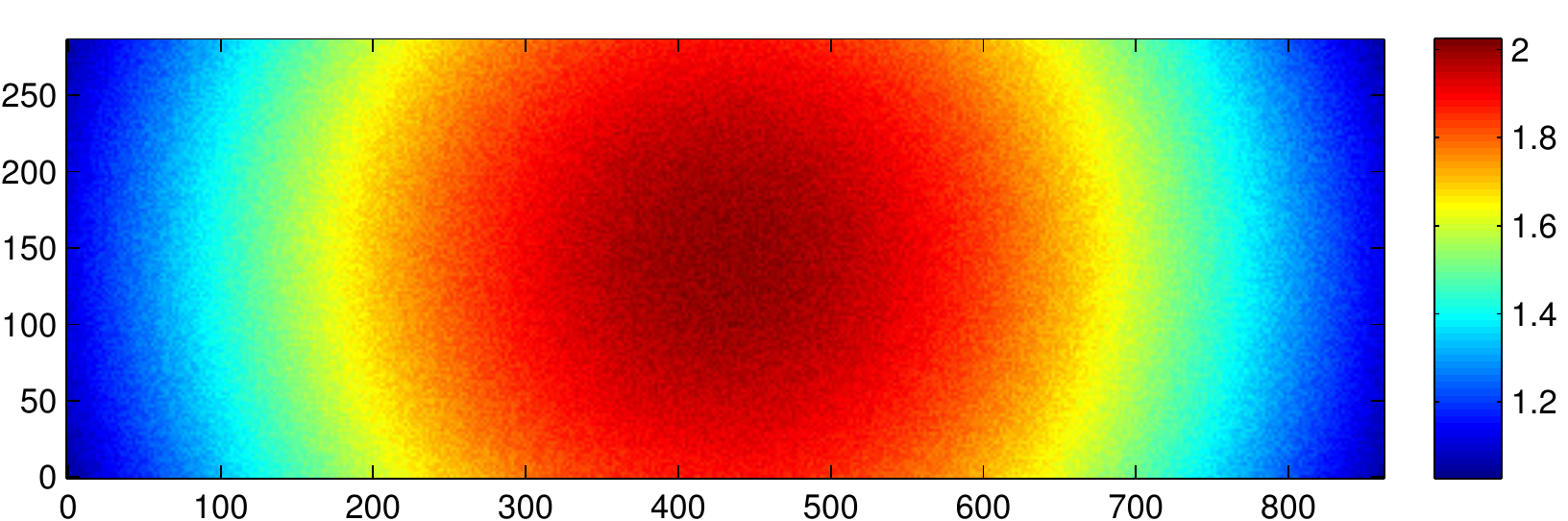}\label{fig:k_gauss}}\\
\subfloat[][]{\includegraphics[width=.45\textwidth]{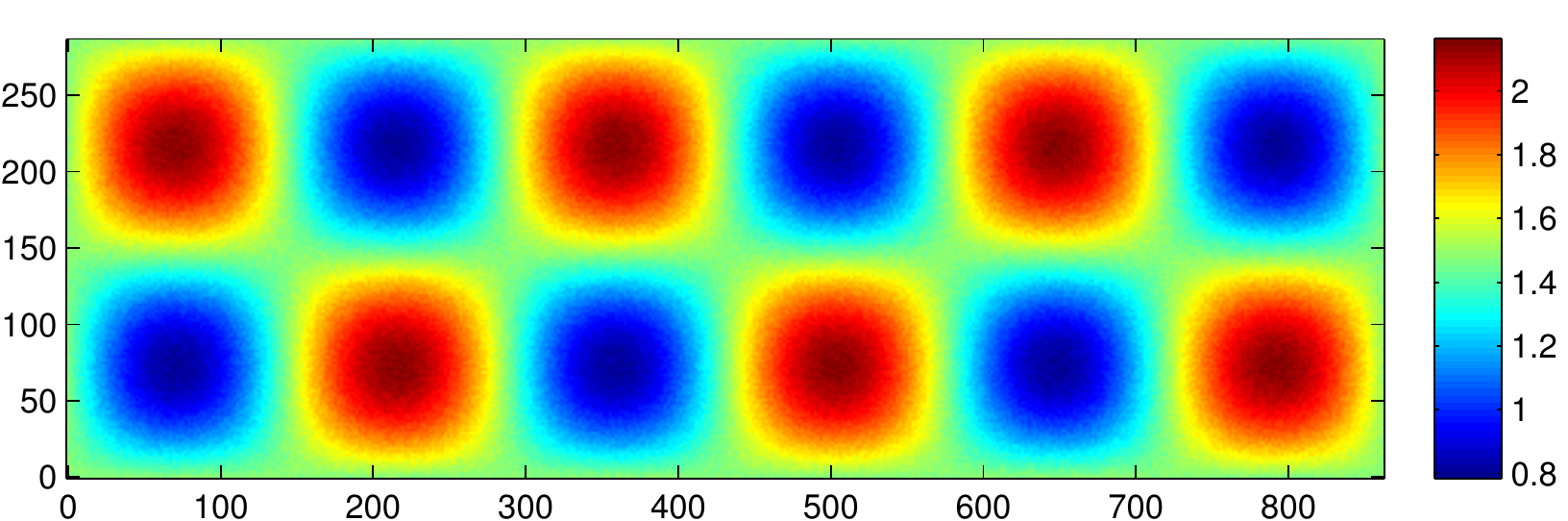}\label{fig:k_sin}}\\
\subfloat[][]{\includegraphics[width=.45\textwidth]{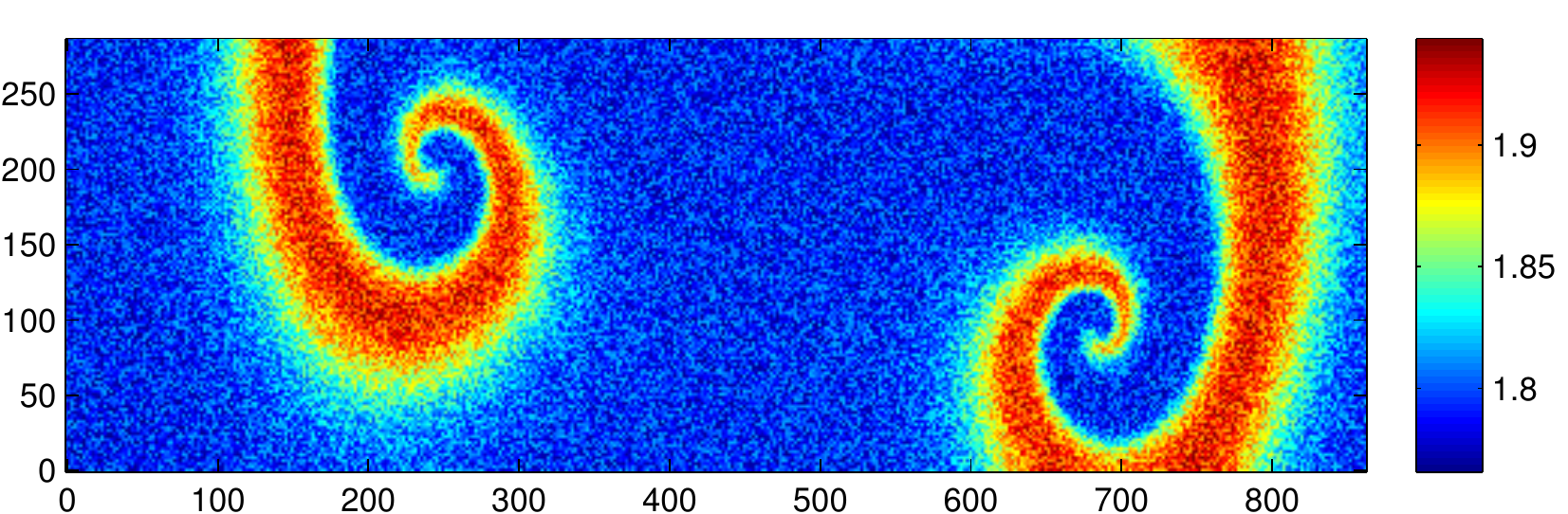}\label{fig:k_swirl}}
\caption{The three different forms spatially dependent parameters used in simulations with apparent noise included. Since $k(x,y)$ and $m(x,y)$ are simply scalar multiples, we plot only $k(x,y)$ for each form. Figure \ref{fig:k_gauss} is described by Eq (\ref{eq:Gauss}). The parameters described by Eq (\ref{eq:Sineplot}) are shown in Figure \ref{fig:k_sin}. Finally, the swirly parameters are shown in Figure \ref{fig:k_swirl}.}
\label{fig:Parameters}
\end{figure}

We discretize the modified system, Eq \eqref{eq:Fish}, with explicit finite differences, using a five-point center difference stencil for spatial derivatives and forward Euler time stepping. The spatial and temporal step sizes are chosen as $dx = 2$ and $dt = 0.2$. The model output $P(x,y,t)$ is treated as an image sequence given by a particular (known) model form but with parameters $k(x,y)$ and $m(x,y)$ and component function $Z(x,y,t)$ to be determined. 

In order to properly mimic our target application of remote sensing oceanographic data of hyperspectral images filtered to reveal plankton blooms, we add random noise and ``moving cloud cover" to the dataset by occluding large proportions of the image from direct observation. Clouds are a natural occurrence when studying a large terrestrial area over several days, and luckily the clouds tend to move.  


\section{Autosynchronization} \label{sec:autosync}
Two model systems are required in order to estimate unknown model states and parameters by autosynchronization, a drive system and response system.  One-way direct replacement and diffusive coupling are combined so that observables are coupled directly into the response model as it evolves. Samples are taken from the drive system,
\begin{equation} \label{eq:DriveModelPDE}
\mathbf{u_t}(x,y,t) = \mathbf{f(u}(x,y),\mathbf{p}(x,y)),
\end{equation} 
with parameters $\mathbf{p}(x,y) \in C^0(\Omega)$ and $u \in H^2 (\Omega)$. A response system is formed,
\begin{equation} \label{eq:ResponseModelPDE}
\mathbf{v_t}(x,y,t) = \mathbf{g(u}(x,y),\mathbf{v}(x,y),\mathbf{q}(x,y)),
\end{equation} 
with $\mathbf{q}(x,y) \in C^0(\Omega)$. 
We formulate an associated system of PDEs for the parameters of Eq \eqref{eq:ResponseModelPDE},
\begin{equation} \label{eq:ParameterEqn}
\mathbf{q_t}(x,y,t) = \mathbf{s(u}(x,y),\mathbf{v}(x,y)),
\end{equation} 
with the goal that $\mathbf{(v,q)} \rightarrow \mathbf{(u,p)}$ as $t \rightarrow \infty$. If successful, the method is called autosynchronization \cite{KP96} since the parameters are evolved deterministically along with the response model.

Generally, some model variables from the drive system need not be sampled. For a two-species system, we write \eq{\textbf{u}(x,y,t) = (u1(x,y,t), u2(x,y,t))^T} and we do not require that \eq{u2(x,y,t)} is sampled. An associated response system \eq{\textbf{v}(x,y,t) = (v1(x,y,t), v2(x,y,t))^T} is built wherein both equations are fed samples from \eq{u1(x,y,t)}. A schematic diagram for this type of simulation might be helpful and is found in Figure \ref{fig:box1}, where dots denote time derivatives. 

\begin{figure}[!h]
\begin{center} 
{\includegraphics[width=.5\textwidth, height = 5cm]{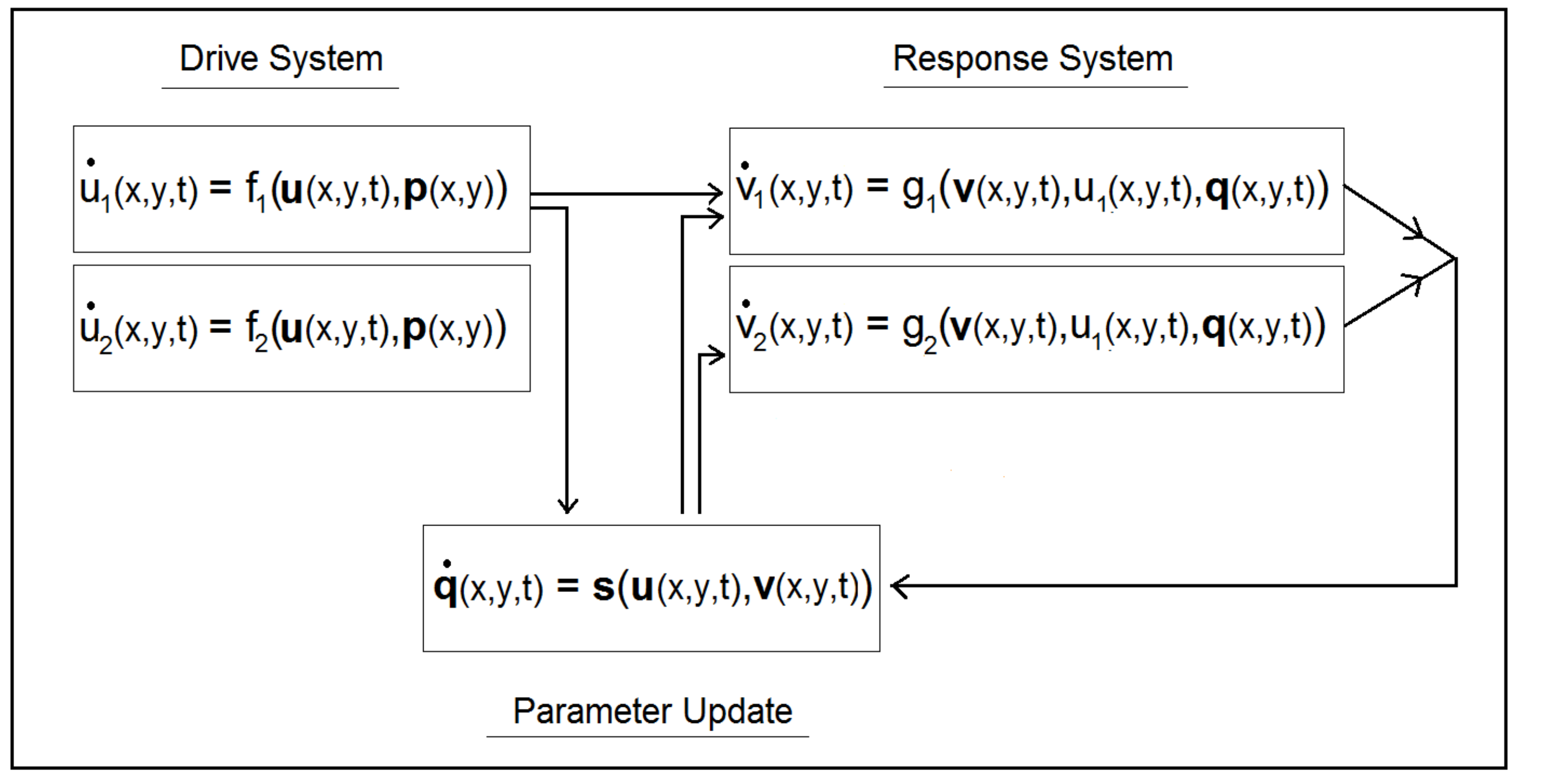}}
\caption{Diagram for autosynchronization of two-component PDE system such as described by Eqs \eqref{eq:DriveModelPDE} - \eqref{eq:ParameterEqn}.}
\label{fig:box1}
\end{center}
\end{figure}

For our synthetic dataset given by Eq \eqref{eq:Fish} with parameters Eq \eqref{eq:Gauss} or Eq \eqref{eq:Sineplot}, we form a response system to be synchronized to the observations as,

\begin{eqnarray}\label{eq:FishPhyto}
\nonumber \frac{\partial \hat{P}}{\partial t} &=& \triangle \hat{P} + \hat{P}(1-\hat{P}) - \frac{\hat{P} \hat{Z}}{\hat{P} + h} + \kappa(P-\hat{P}), \ \ \ \textrm{and} \\
 \frac{\partial\hat{Z}}{\partial t} &=& \triangle \hat{Z} + \hat{k}\frac{\hat{P} \hat{Z}}{P+h} - \hat{m}\hat{Z},
\end{eqnarray}
where we assume $\hat{P}(x,y,0) \neq P(x,y,0)$, $\hat{Z}(x,y,0) \neq Z(x,y,0)$, $\hat{k}(x,y,0) \neq k(x,y)$, and $\hat{m}(x,y,0) \neq m(x,y)$.

Parameters are updated as diffusively coupled PDEs during the synchronization process as,
\begin{eqnarray}\label{eq:ParUpdate}
\nonumber \frac{\partial \hat{k}}{\partial t} &=& -s(P - \hat{P}), \ \ \ \textrm{and} \\
 \frac{\partial\hat{m}}{\partial t} &=& -s(P - \hat{P})\hat{P}, 
\end{eqnarray} 
where $s = 30$ and $\kappa = 2.4$ are chosen for specificity and for which we observe good convergence results. For these experiments, we sample the drive system at every time step, but note that a larger sampling time will work \cite{kramer2013spatially}. The parameter equations are evolved simultaneously by Eq \eqref{eq:FishPhyto} with a forward Euler discretization and the same time step. As we vary $s$ and $\kappa$, autosynchronization may fail as commonly observed with diffusively-coupled systems. 
 
To begin the simulation, parameters are initialized as the constant function, e.g. $\hat{k}(x,y,0) = 5$ and $\hat{m}(x,y,0) = 5$. We evolve Eq \eqref{eq:Fish} forward and count the model output as observed data. Initial conditions for the response system are $\hat{P}(x,y,0) = 2$ and $\hat{Z}(x,y,0) = 2$. Furthermore, to avoid values outside the normal range of Eq \eqref{eq:Fish}, we enforce that \\[-1.0cm] 
\begin{center}
\begin{minipage}[h]{4in}
\begin{parcolumns}[colwidths={1=1 in},nofirstindent]{2}
\colchunk[1]{\begin{eqnarray}
\nonumber \hat{P} = \Bigg\{ \begin{matrix} \hat{P}&:& 0 < \hat{P} < 2 \\0&:& \hat{P} \leq 0\\2&:& \hat{P} \geq 2 \end{matrix} \textrm{   \hskip .4cm   and}
\end{eqnarray}}

\colchunk[2]{\begin{eqnarray}
\nonumber \hat{Z} = \Bigg\{ \begin{matrix} \hat{Z}&:& 0 < \hat{Z} < 2 \\0&:& \hat{Z} \leq 0\\2&:& \hat{Z} \geq 2 \end{matrix}
\end{eqnarray}}
\end{parcolumns}
\end{minipage}
\end{center}

\noindent $\forall x,y \in \Omega$ during the simulation.  As noted above, autosynchronization is observed for the test set of parameters in Figure \ref{fig:Parameters}  and the spatial inhomogeneities in each case are effectively resolved. We emphasize that zooplankton are not observed in Eq \eqref{eq:FishPhyto}-\eqref{eq:ParUpdate}.


\section{Hidden Data}

Ocean-observing satellite imagery often includes significant amounts of cloud cover \cite{martin2014introduction}. In other words, a large fraction of that data may be occluded. Furthermore, we have found that \textit{level 2} mapped and processed images may include striping or other defects from projecting a sphere onto a uniform grid. The lack of data presents a challenge to data assimilation and model filtering by synchronization methods. Suppose $\omega \subset \Omega$ is the set of unobservable data. We allow for $\omega = \omega(x,y,t)$ so that the set of unobservables varies with space and time like a cloud. We consider a simple case where the dynamics of $\omega(x,y,t)$ are governed by the advection equation
\begin{equation}
\nonumber \frac{\partial \omega}{\partial t} + \nu \frac{\partial \omega}{\partial x} = 0,
\end{equation}
with periodic boundary conditions, so that clouds move in the x-direction with speed $\nu$. We couple the systems only on the complement of $\omega$. That is, we turn the driving signal off when the image is unobservable, allowing the two systems to oscillate independently, and switch it on after the clouds have passed. We do this only in the subregion $\omega \subset \Omega$ that is unobservable in order that data contained in the complement of $\omega$ may continue to be driven by observables toward the synchronization manifold.

Here we build on the method described in section \ref{sec:autosync}, where zooplankton densities and model parameters are estimated by observing solely the phytoplankton. Now we observe phytoplankton and clouds. However, if we couple at every spatial grid point, the synchronization manifold is de-stabilized by incident cloud coverage. With a large enough amount of cloud coverage over $\Omega$, the systems fail to synchronize. We say large enough in deference to the case where the occluded region is small enough such that diffusion allows information to pass into any hidden regions. 

As a remedy we allow the drive and response models to oscillate independently, or uncoupled, while the drive model is hidden by clouds. The pixels representing cloud cover in remote sensing data are typically set to some large fixed integer, $I$. We represent this formally,

\begin{eqnarray}\label{eq:Fish_Sync_Phyto_clouds}
\nonumber \frac{\partial \hat{P}}{\partial t} &=& \triangle \hat{P} + \hat{P}(1-\hat{P}) - \frac{\hat{P} \hat{Z}}{\hat{P} + h} + \kappa  (H[P] - \hat{P}), \ \ \ \textrm{and}\\
 \frac{\partial\hat{Z}}{\partial t} &=& \triangle \hat{Z} + k\frac{\hat{P} \hat{Z}}{\hat{P} +h} - m\hat{Z},
\end{eqnarray}
 where $H[P]$ represents a \textit{switching} function given by
\begin{minipage}[h]{=\textwidth}
\begin{parcolumns}[colwidths={1=3 in},nofirstindent]{2}
\colchunk[1]{\begin{eqnarray}
\nonumber H[P] = \left\{ \begin{matrix} \hat{P}&,&  P = I \\P&,& P \neq I. \end{matrix} \right.
\end{eqnarray}}
\end{parcolumns}
\end{minipage}
\\

The form of response model switches off the coupling when a cloud mask is detected in the image and allows the systems to oscillate independently in the corresponding pixels, while being driven over pixels that are observed. Eq (\ref{eq:Fish_Sync_Phyto_clouds}) is slightly different from temporal subsampling of data, where models are not coupled for a given number of time steps. Here the models are always coupled somewhere in $\Omega$, which is determined by time-varying clouds. 

\begin{figure}
\centering
\subfloat[][]{\includegraphics[width=.25\textwidth]{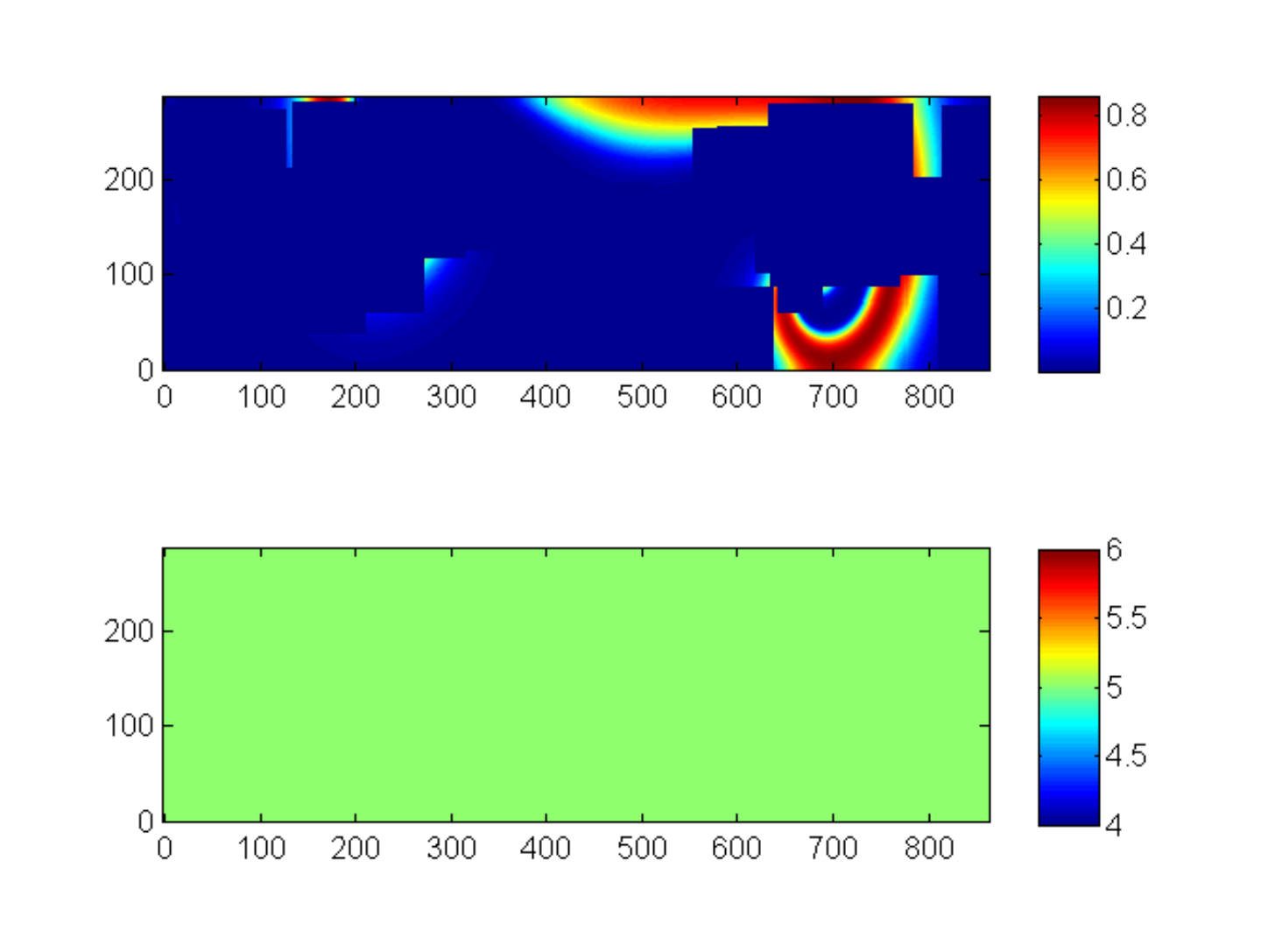}\label{fig:Sync_phyto_clouds_IC}}
\subfloat[][]{\includegraphics[width=.25\textwidth]{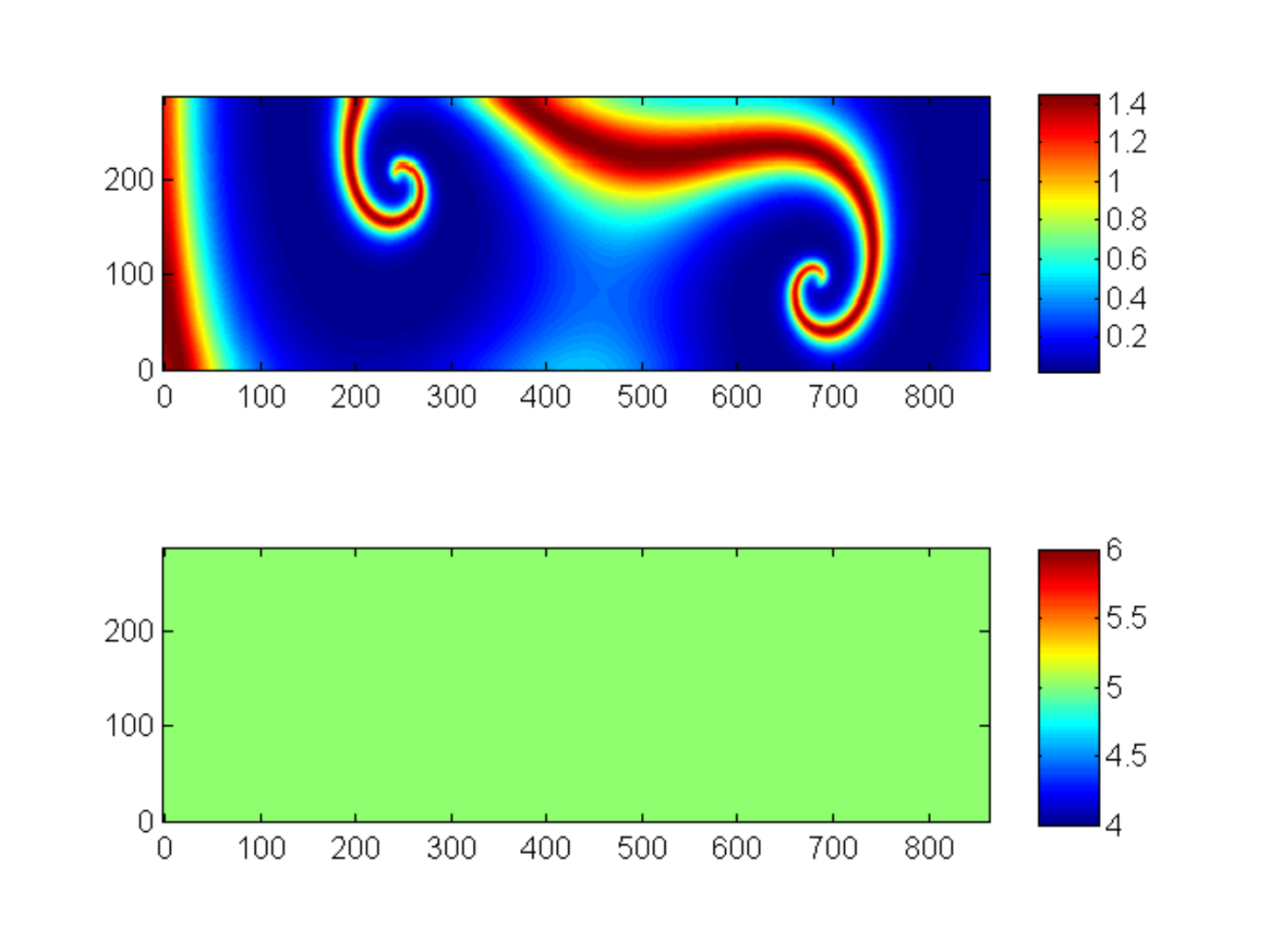}\label{fig:Sync_zoo_clouds_IC}}\\[-.4cm]
\subfloat[][]{\includegraphics[width=.25\textwidth]{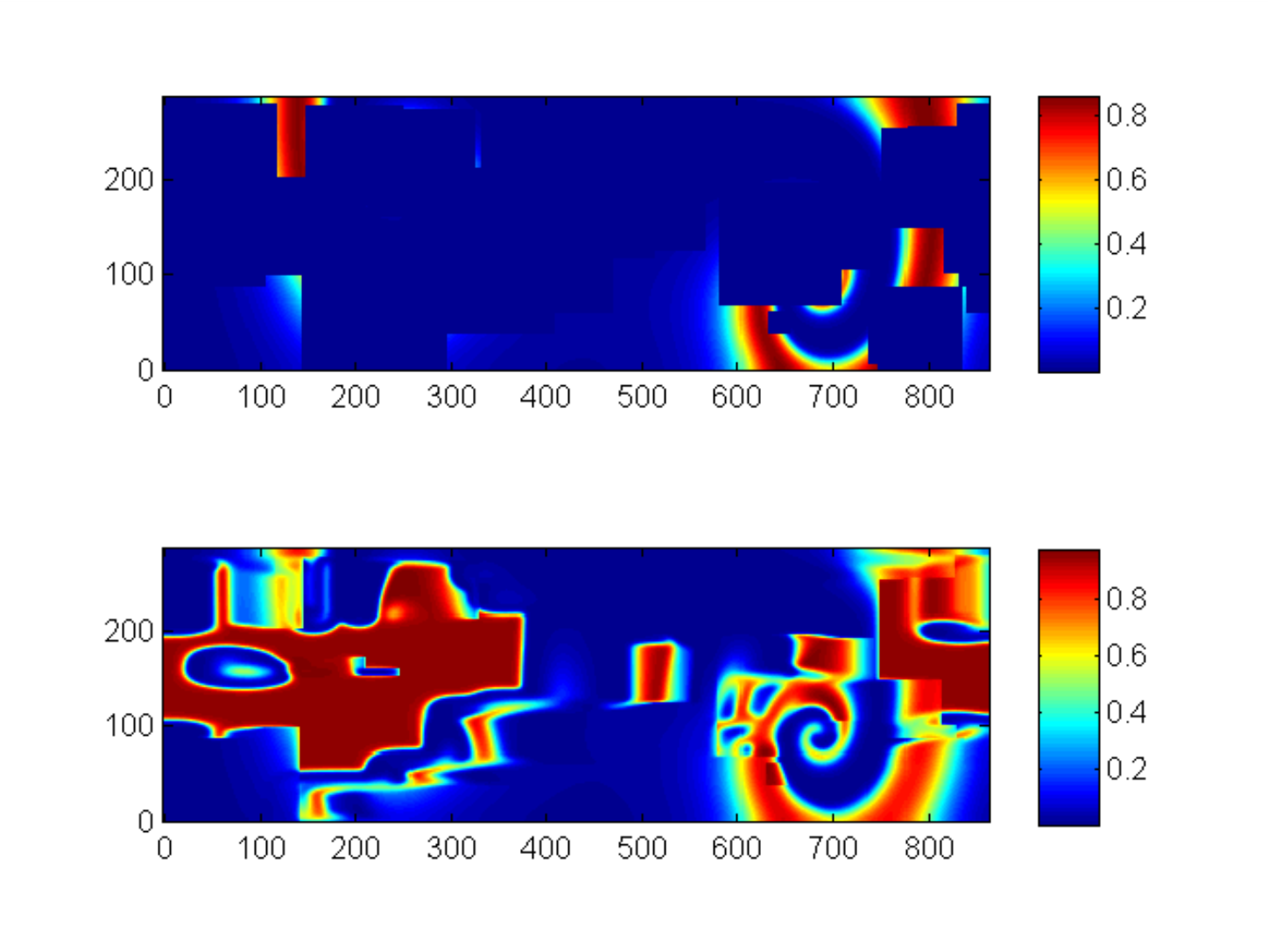}\label{fig:Sync_phyto_clouds_100}}
\subfloat[][]{\includegraphics[width=.25\textwidth]{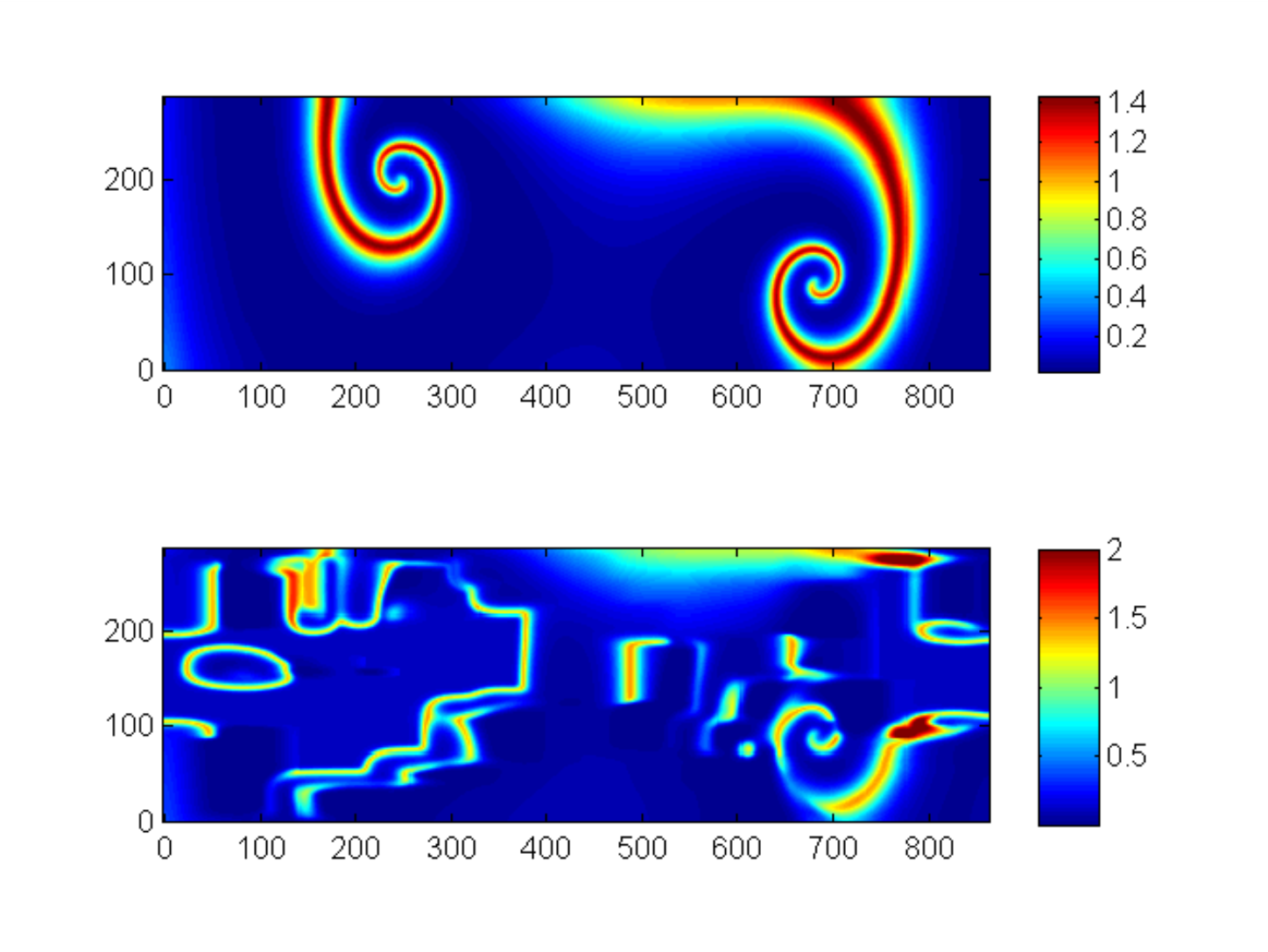}\label{fig:Sync_zoo_clouds_100}}\\[-.4cm]
\subfloat[][]{\includegraphics[width=.25\textwidth]{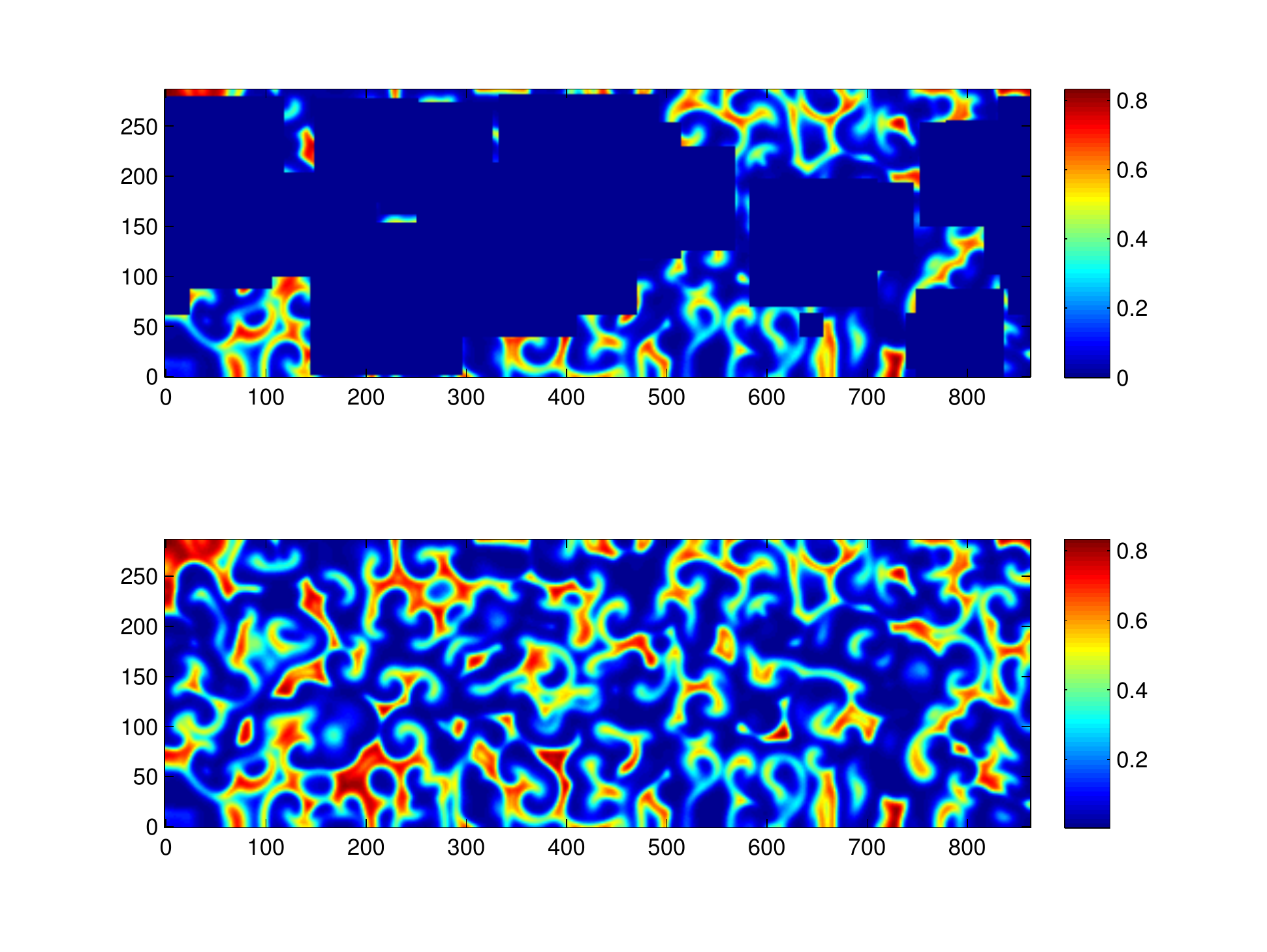}\label{fig:Sync_phyto_clouds_FC}}
\subfloat[][]{\includegraphics[width=.25\textwidth]{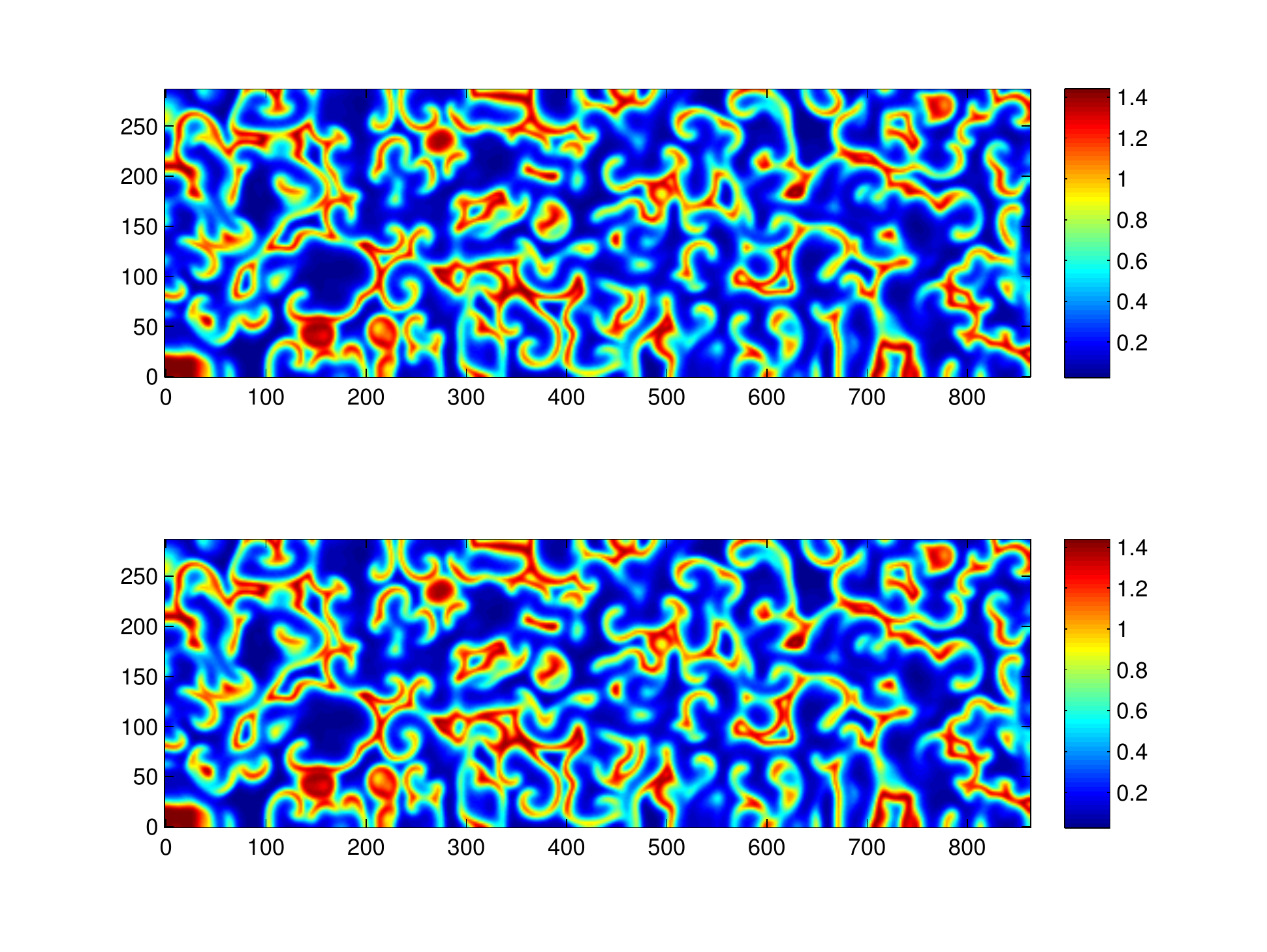}\label{fig:Sync_zoo_clouds_FC}}
\caption{Synchronization of response system shown at $t=0$, $t=20$, and $t = 12000$. Here $65.8 \%$ of $\Omega$ is hidden at any point in time from clouds, however identical synchronization is observed. Each figure shows drive (top) and response (bottom) pairs. $P(x,y,0)$ and $\hat{P}(x,y,0)$ in \ref{fig:Sync_phyto_clouds_IC}, $P(x,y,20)$ and $\hat{P}(x,y,20)$ in \ref{fig:Sync_phyto_clouds_100}, and $P(x,y,12000)$ and $\hat{P}(x,y,12000)$ in \ref{fig:Sync_phyto_clouds_FC}. $Z(x,y,0)$ and $\hat{Z}(x,y,0)$ in \ref{fig:Sync_zoo_clouds_IC}, $Z(x,y,20)$ and $\hat{Z}(x,y,20)$ in \ref{fig:Sync_zoo_clouds_100}, and $Z(x,y,12000)$ and $\hat{Z}(x,y,12000)$ in \ref{fig:Sync_zoo_clouds_FC}.}
\label{fig:Clouds_Phyto}
\end{figure}

\begin{figure}
\centering
\includegraphics[width = .45 \textwidth]{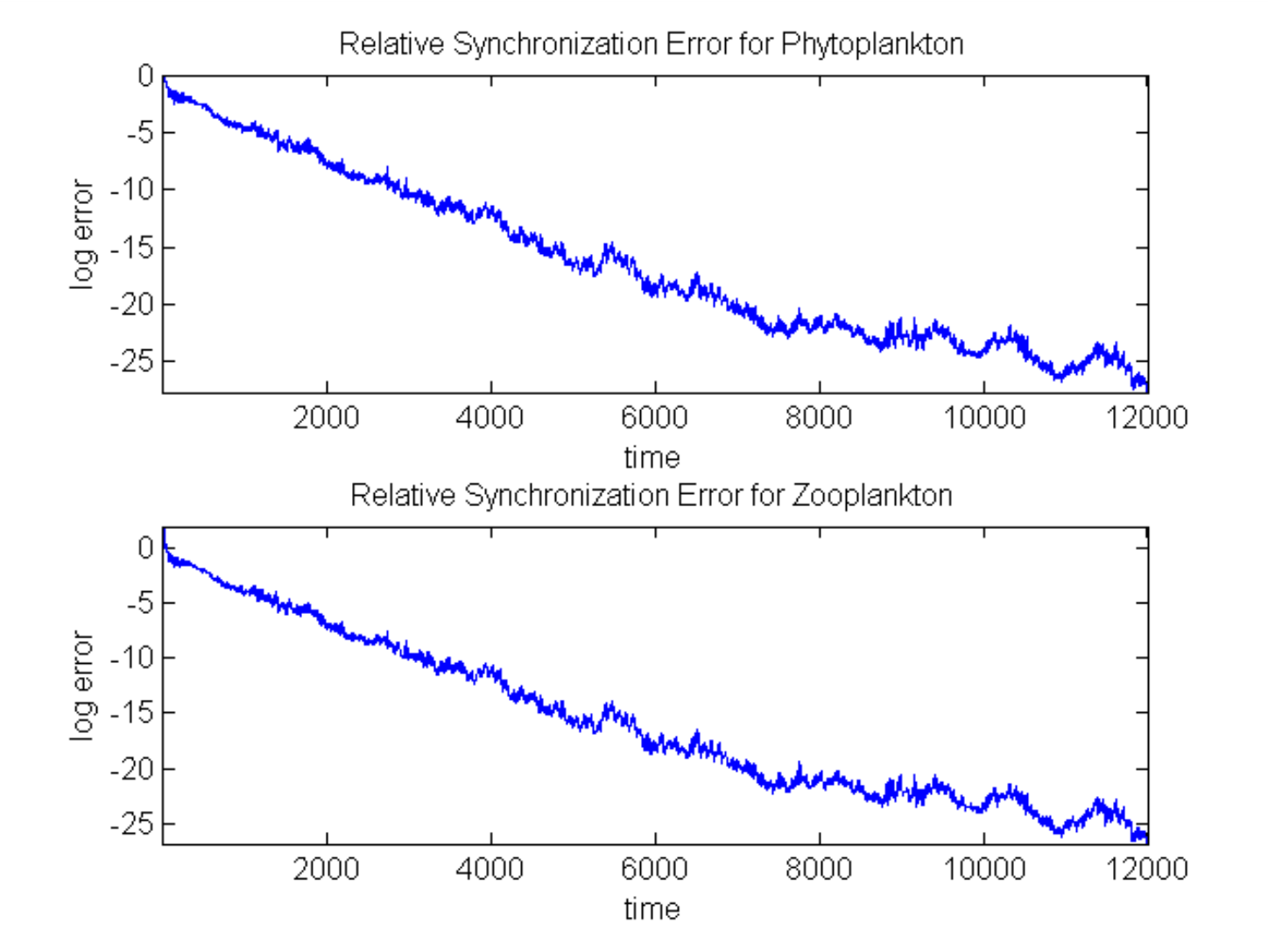}\label{Sync_Error_Clouds}
\caption{Globally-averaged relative synchronization errors. Errors given by simulation shown in Figure \ref{fig:Clouds_Phyto} decrease to less than $2.6 \times 10^{-12}$ despite ever-present clouds.}
\label{fig:Clouds_Sync_Error} 
\end{figure}

We first demonstrate model state synchronization given occluded data before addressing the estimation of model parameters. Let Eq (\ref{eq:Fish}) be the drive model and Eq (\ref{eq:Fish_Sync_Phyto_clouds}) be the response model. Figure \ref{fig:Clouds_Phyto} represents a partially observed dataset from Eq (\ref{eq:Fish}), with a field of 30 randomly placed synthetic clouds evolving from left to right with periodic boundary conditions resulting in $65.8 \%$ of $\Omega$ occluded at all times. The clouds repeatedly scroll from left to right and parts of the image are always occluded, but every element in the domain is eventually driven, causing the drive and response to systems to synchronize. The response system is initialized by $\hat{P}(x,y,0) = 2$ and $\hat{Z}(x,y,0) = 2$, and we choose $\kappa = 2.6$. 

Once synchronized, even hidden phytoplankton are revealed for initializing short-term forecasts, demonstrating the utility of this result. Figure \ref{fig:Clouds_Sync_Error} demonstrates that despite $65.8 \%$ of the drive system hidden, the two PDE systems eventually evolve toward identical synchronization. In Figures \ref{fig:Sync_phyto_clouds_FC} and \ref{fig:Sync_zoo_clouds_FC} nearly all evidence of clouds is ``synchronized away" from the response system and the globally averaged error between the two has been driven to be less than $2.6 \times 10^{-12}$. We remark that the choice of coupling strength, $\kappa$, varies with the amount of data occluded. 

Given a model form, we advance the method to sample a single species toward parameter estimation and nonlinear data assimilation for a two-species PDE model, regardless of clouds. That is, by stating the response system
\begin{eqnarray}\label{eq:Fish_Auto_Phyto_Clouds}
\nonumber \frac{\partial \hat{P}}{\partial t} &=& \triangle \hat{P} + \hat{P}(1-\hat{P}) - \frac{\hat{P} \hat{Z}}{\hat{P} + h} + \kappa(H[P]-\hat{P}),\\
\nonumber \frac{\partial\hat{Z}}{\partial t} &=& \triangle \hat{Z} + \hat{k}\frac{\hat{P} \hat{Z}}{H[P]+h} - \hat{m}\hat{Z},\\
\nonumber \frac{\partial \hat{k}}{\partial t} &=& s_1(H[P] - \hat{P}), \ \ \ \textrm{and} \\
 \frac{\partial\hat{m}}{\partial t} &=& s_2(H[P] - \hat{P})\hat{P},
\end{eqnarray}
where $k(x,y)$, $m(x,y)$, $Z(x,y,t)$, and $P(x,y,t)|_\omega$  are to be estimated by $\hat{k}(x,y)$, $\hat{m}(x,y)$, $\hat{Z}(x,y,t)$, and $\hat{P}(x,y,t)$ by sampling \textit{only} $P(x,y,t)|_{\omega^C}$. As before, the coupling is turned off completely for the pixels on which clouds are detected. The drive system is still Eq (\ref{eq:Fish}), but we allow for spatially dependent model parameters $k(x,y), m(x,y)$ in the form of Figure \ref{fig:Parameters}. For robustness, we consider that the model parameters need not have the same functional form. We choose $m(x,y)$ defined by Eq \eqref{eq:Sineplot} and $k(x,y)$ with form shown in Figure \ref{fig:k_swirl} and also add random noise to both parameters.

\begin{figure}
\centering
\subfloat[][]{\includegraphics[width=.25\textwidth]{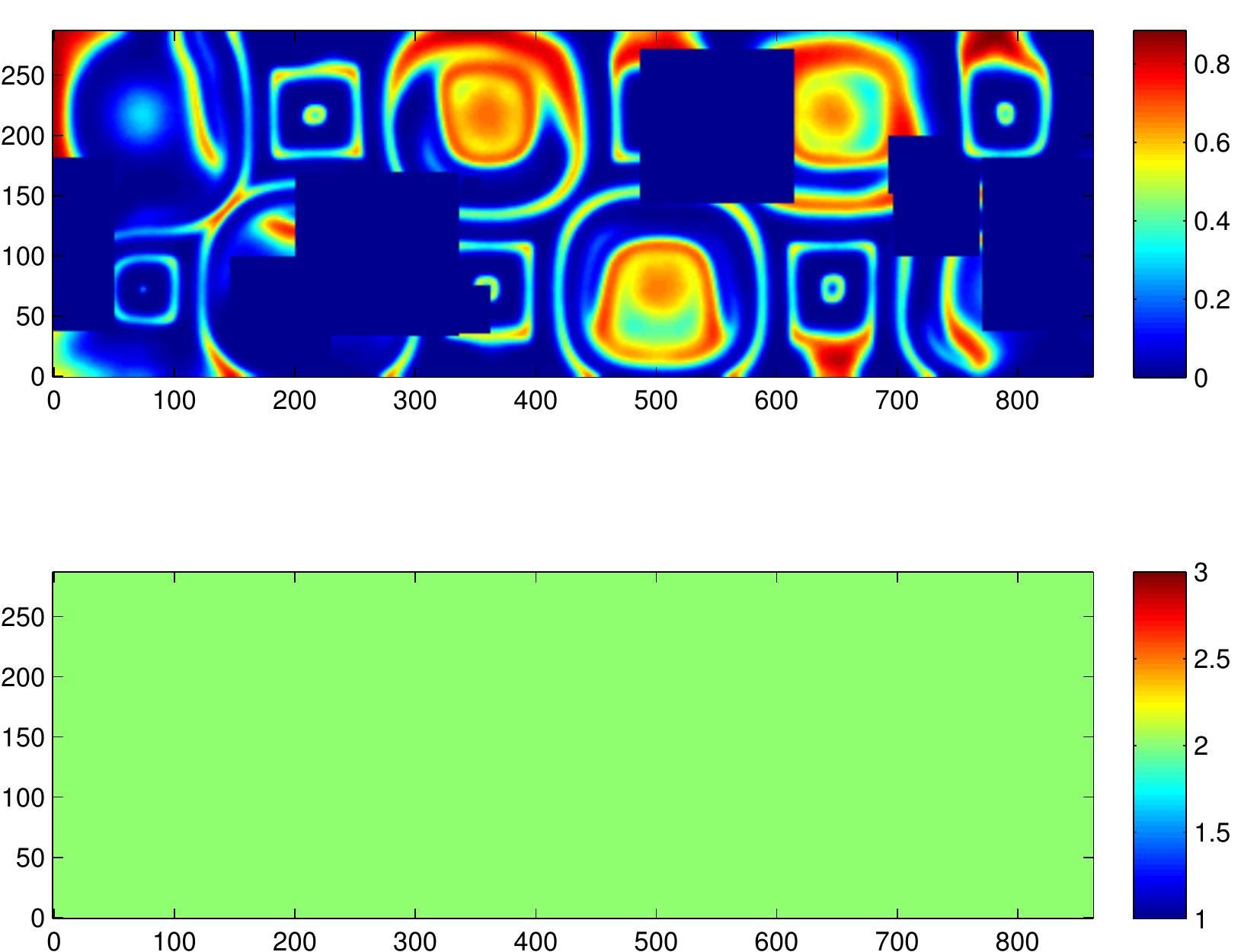}\label{fig:Autosync_phyto_clouds_IC}}
\subfloat[][]{\includegraphics[width=.25\textwidth]{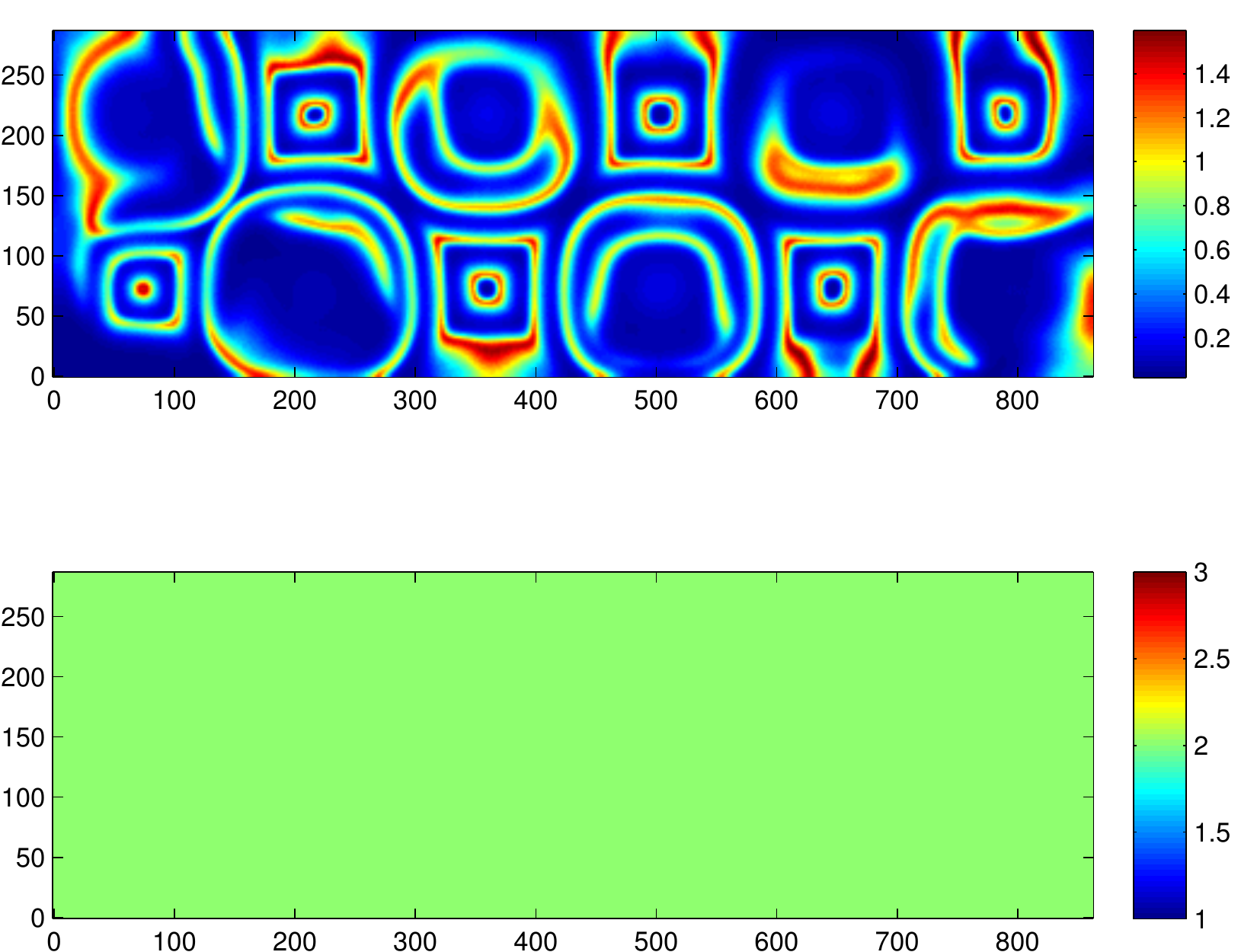}\label{fig:Autosync_zoo_clouds_IC}}\\[-.4cm]
\subfloat[][]{\includegraphics[width=.25\textwidth]{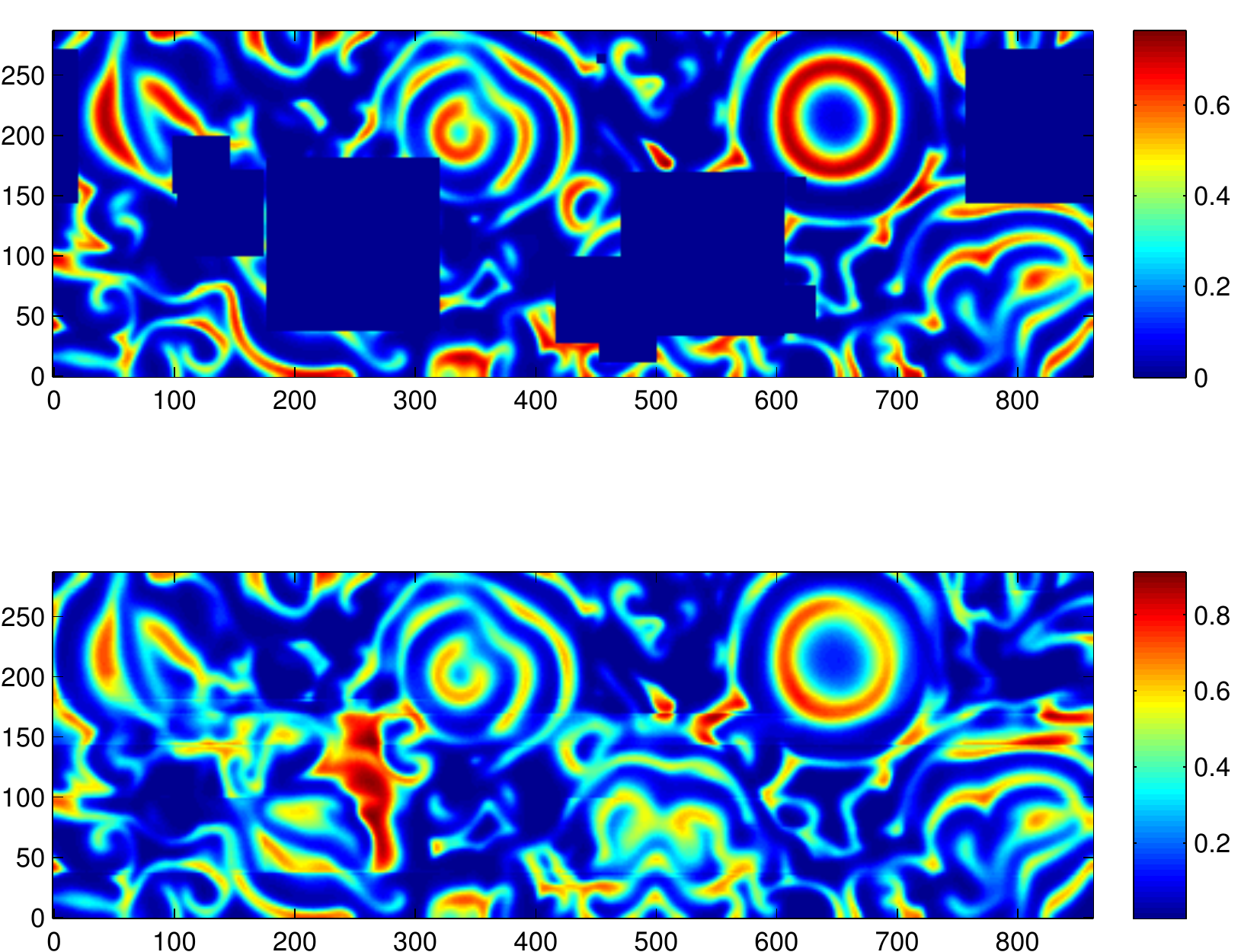}\label{fig:Autosync_phyto_clouds_1000}}
\subfloat[][]{\includegraphics[width=.25\textwidth]{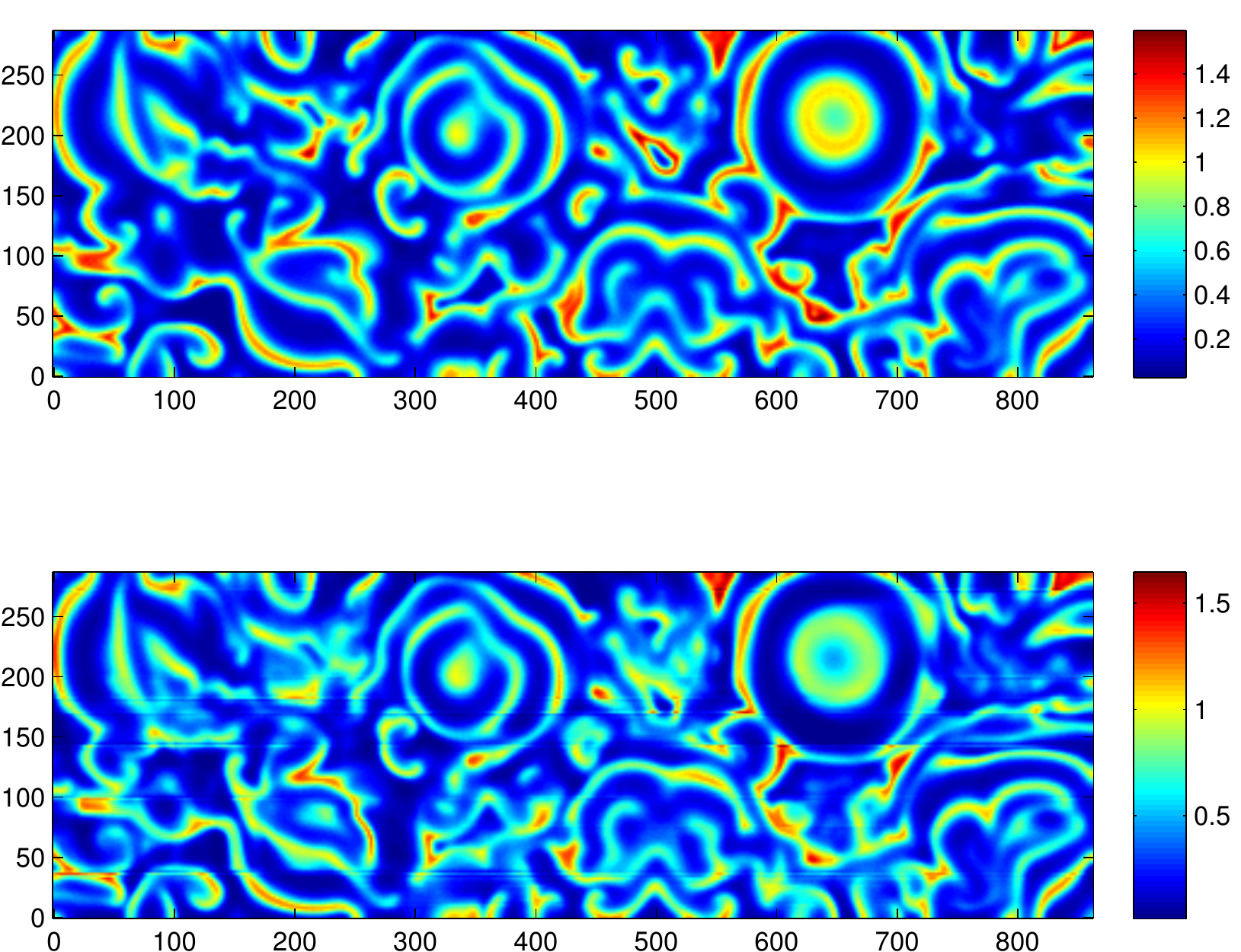}\label{fig:Autosync_zoo_clouds_1000}}\\[-.4cm]
\subfloat[][]{\includegraphics[width=.25\textwidth]{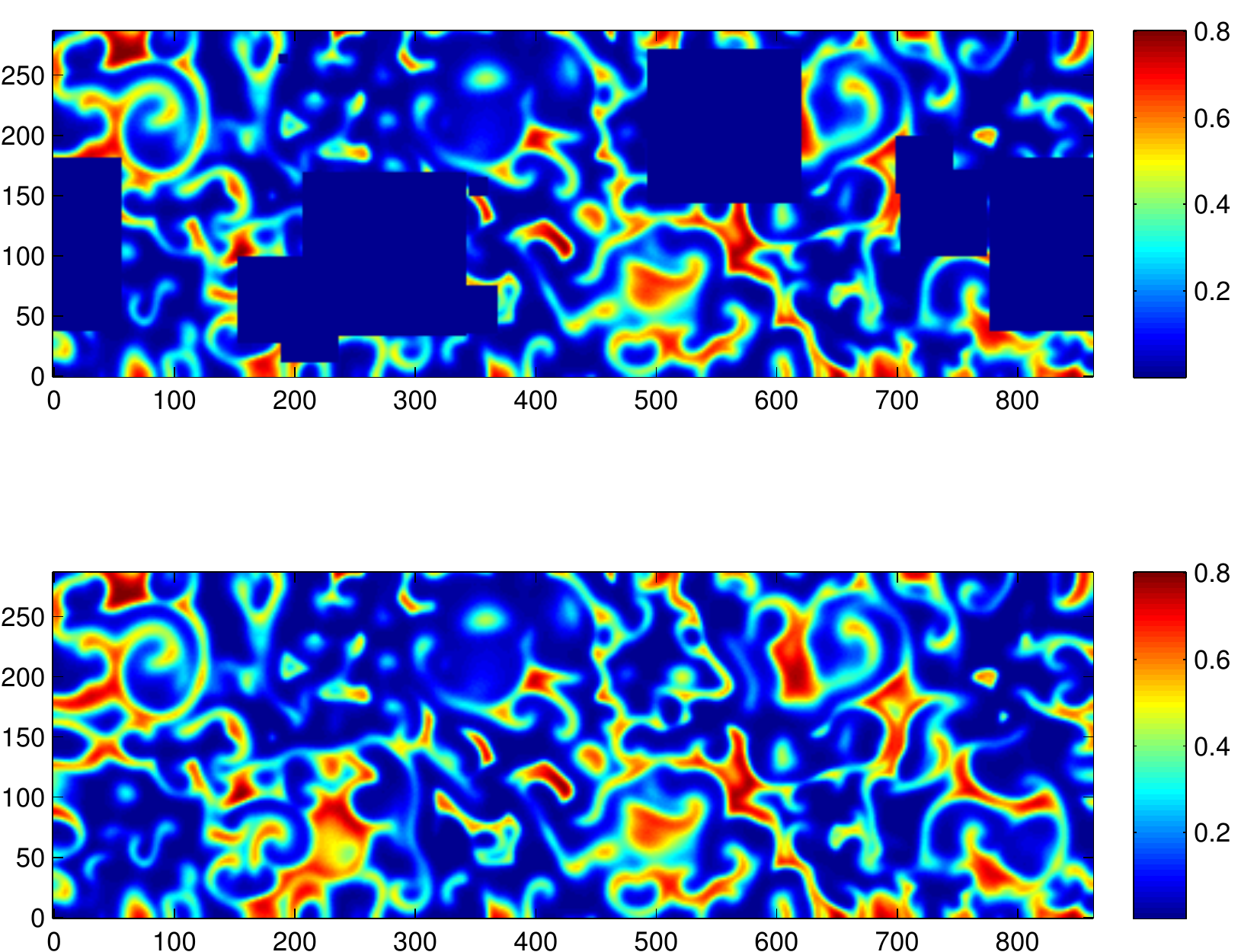}\label{fig:Autosync_phyto_clouds_FC}}
\subfloat[][]{\includegraphics[width=.25\textwidth]{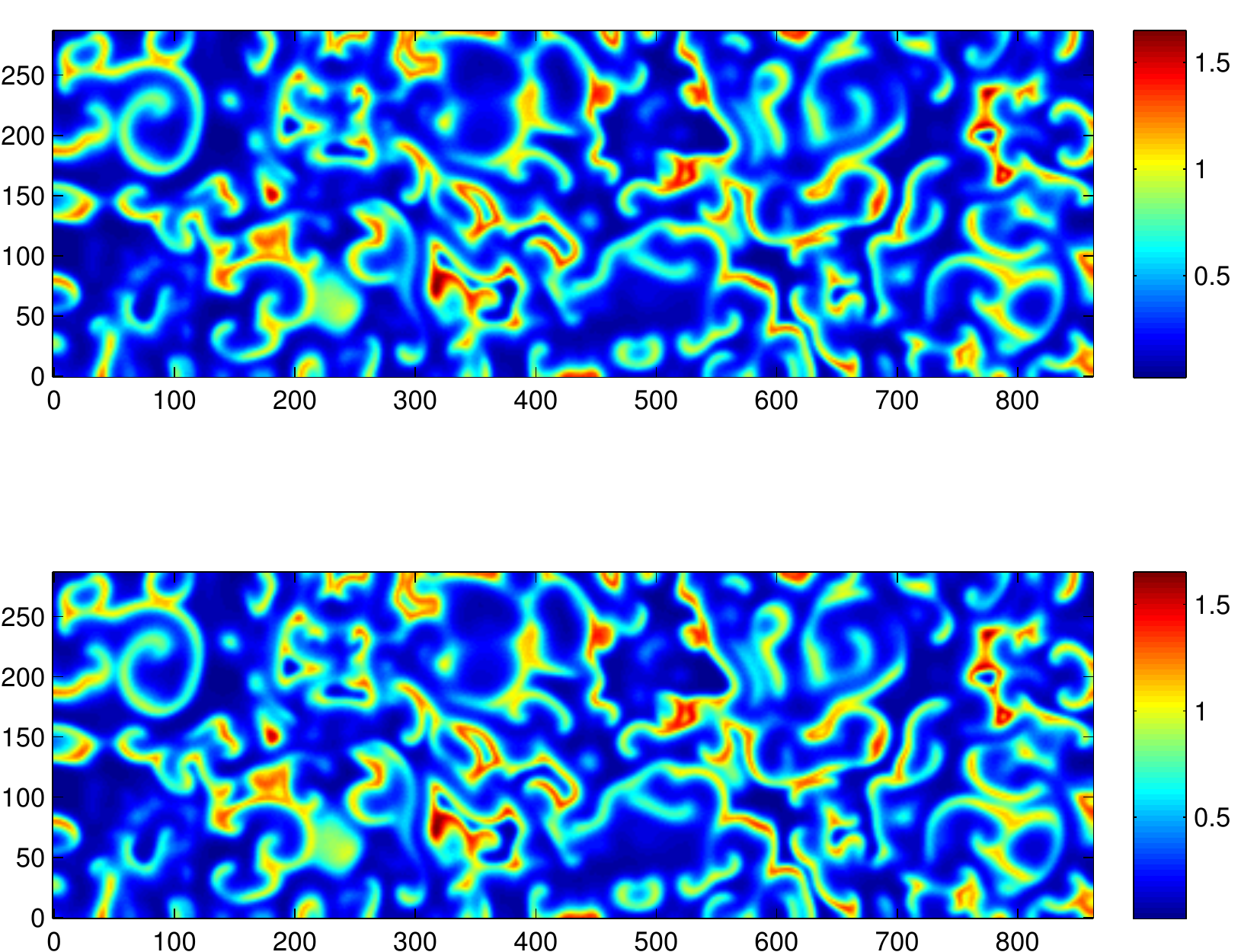}\label{fig:Autosync_zoo_clouds_FC}}
\caption{Autosynchronization of species with $25.5 \%$ of $\Omega$ is hidden at any point in time from clouds, however autosynchronization is observed. Each figure shows drive (top) and response (bottom) pairs. $P(x,y,0)$ and $\hat{P}(x,y,0)$ in \ref{fig:Autosync_phyto_clouds_IC}, $P(x,y,200)$ and $\hat{P}(x,y,200)$ in \ref{fig:Autosync_phyto_clouds_1000}, and $P(x,y,8563)$ and $\hat{P}(x,y,8563)$ in \ref{fig:Autosync_phyto_clouds_FC}. $Z(x,y,0)$ and $\hat{Z}(x,y,0)$ in \ref{fig:Autosync_zoo_clouds_IC}, $Z(x,y,200)$ and $\hat{Z}(x,y,200)$ in \ref{fig:Autosync_zoo_clouds_1000}, and $Z(x,y,8563)$ and $\hat{Z}(x,y,8563)$ in \ref{fig:Autosync_zoo_clouds_FC}.}
\label{fig:Clouds_Auto_Species}
\end{figure}

\begin{figure}
\centering
\subfloat[][]{\includegraphics[width=.25\textwidth]{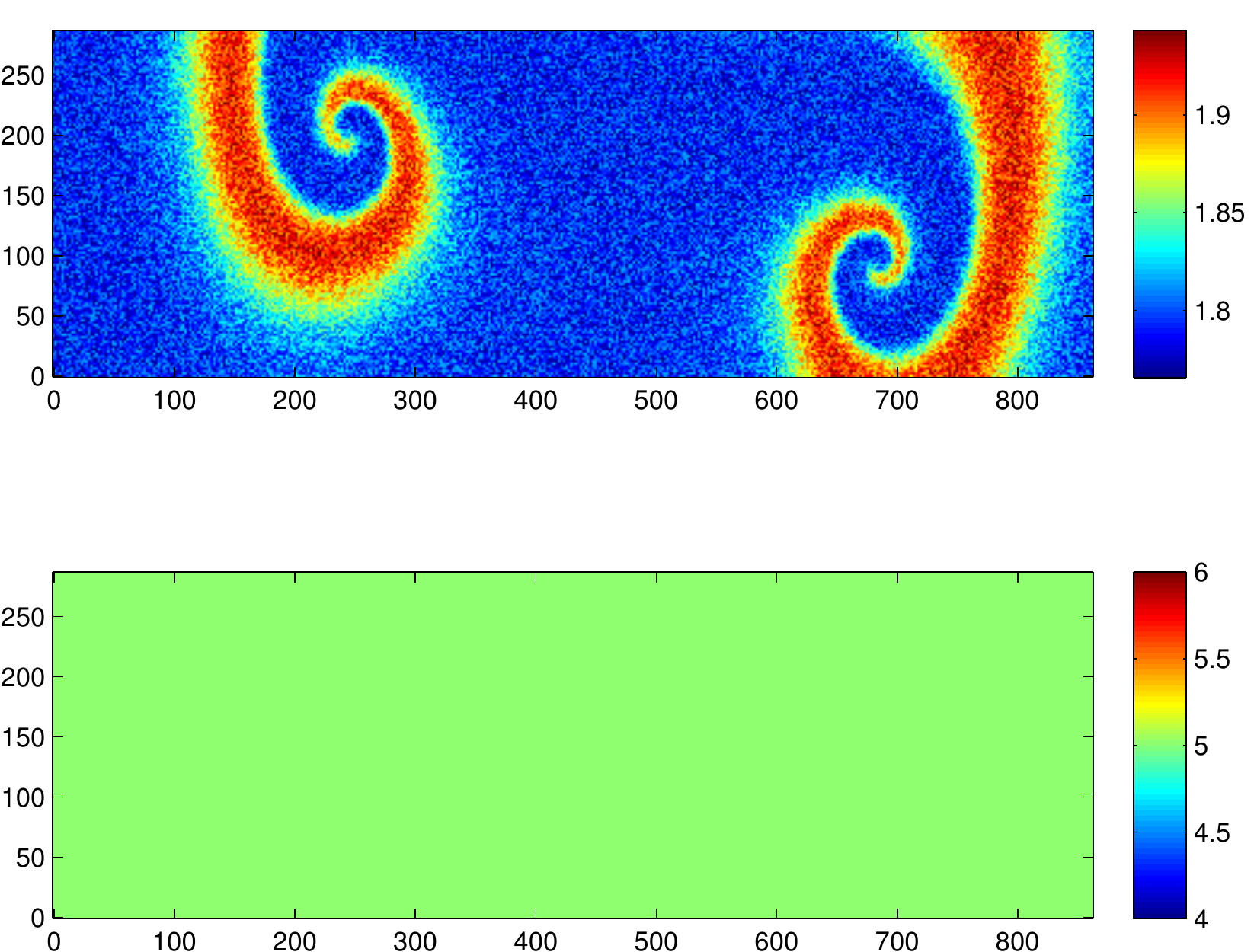}\label{fig:Autosync_k_clouds_IC}}
\subfloat[][]{\includegraphics[width=.25\textwidth]{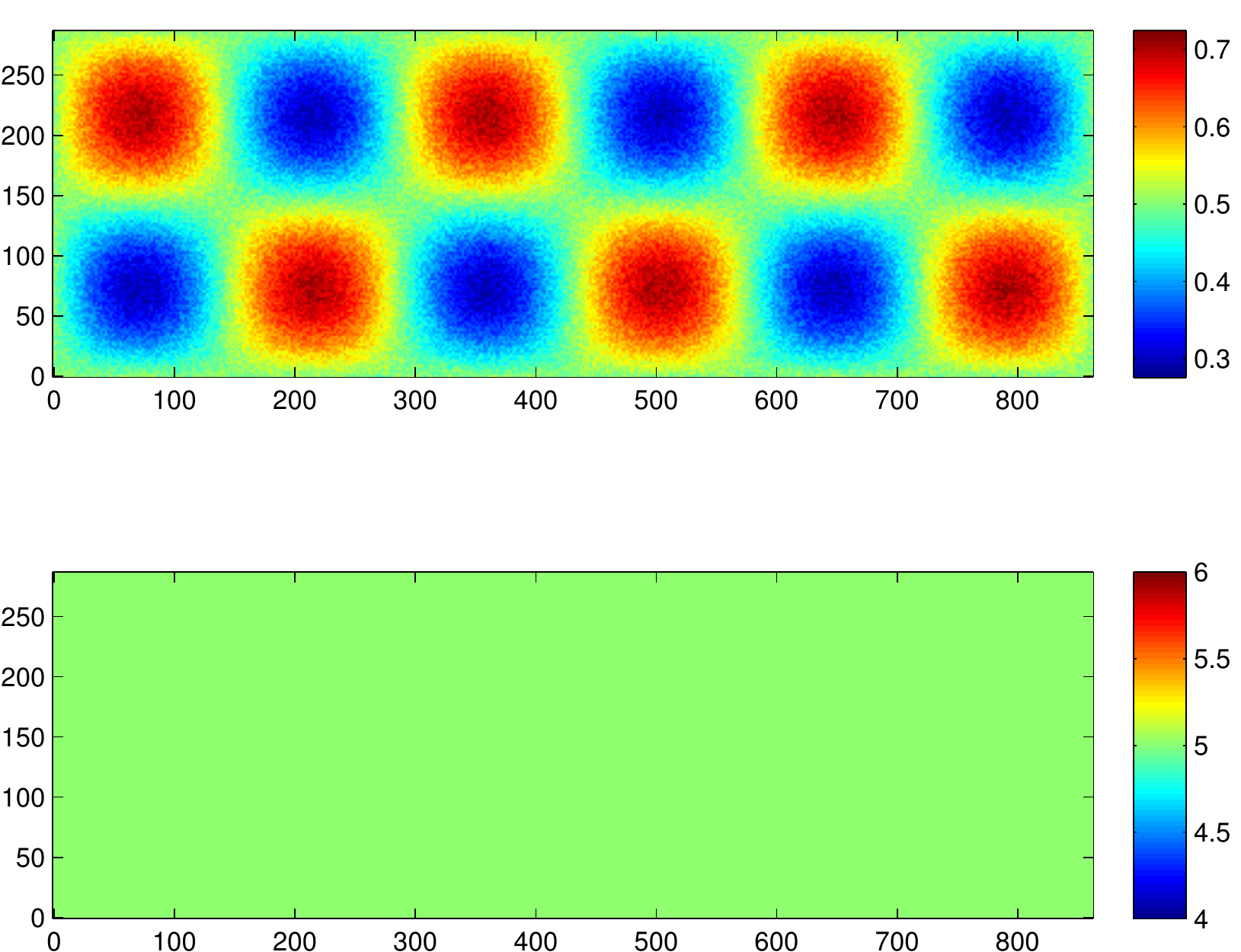}\label{fig:Autosync_m_clouds_IC}}\\[-.4cm]
\subfloat[][]{\includegraphics[width=.25\textwidth]{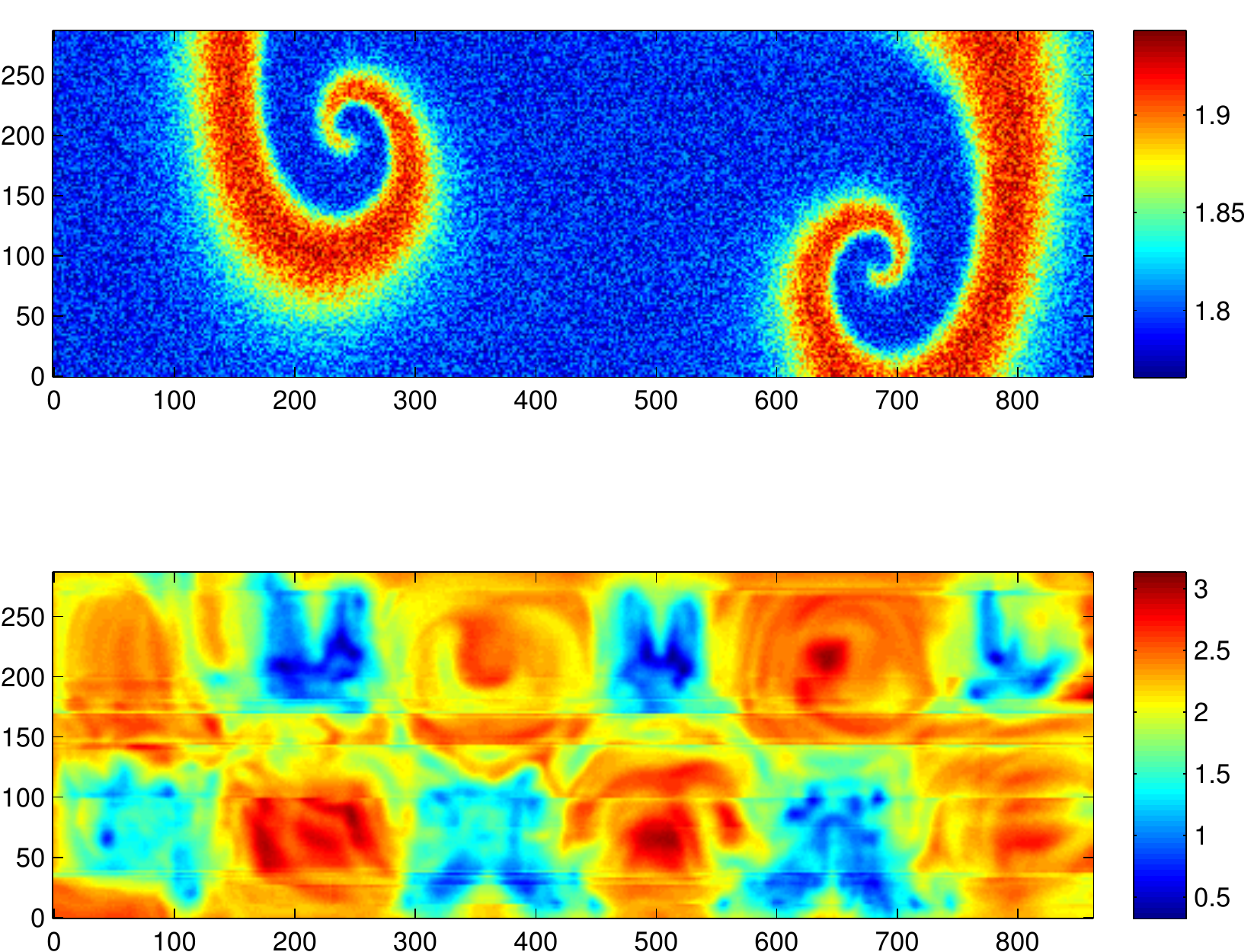}\label{fig:Autosync_k_clouds_1000}}
\subfloat[][]{\includegraphics[width=.25\textwidth]{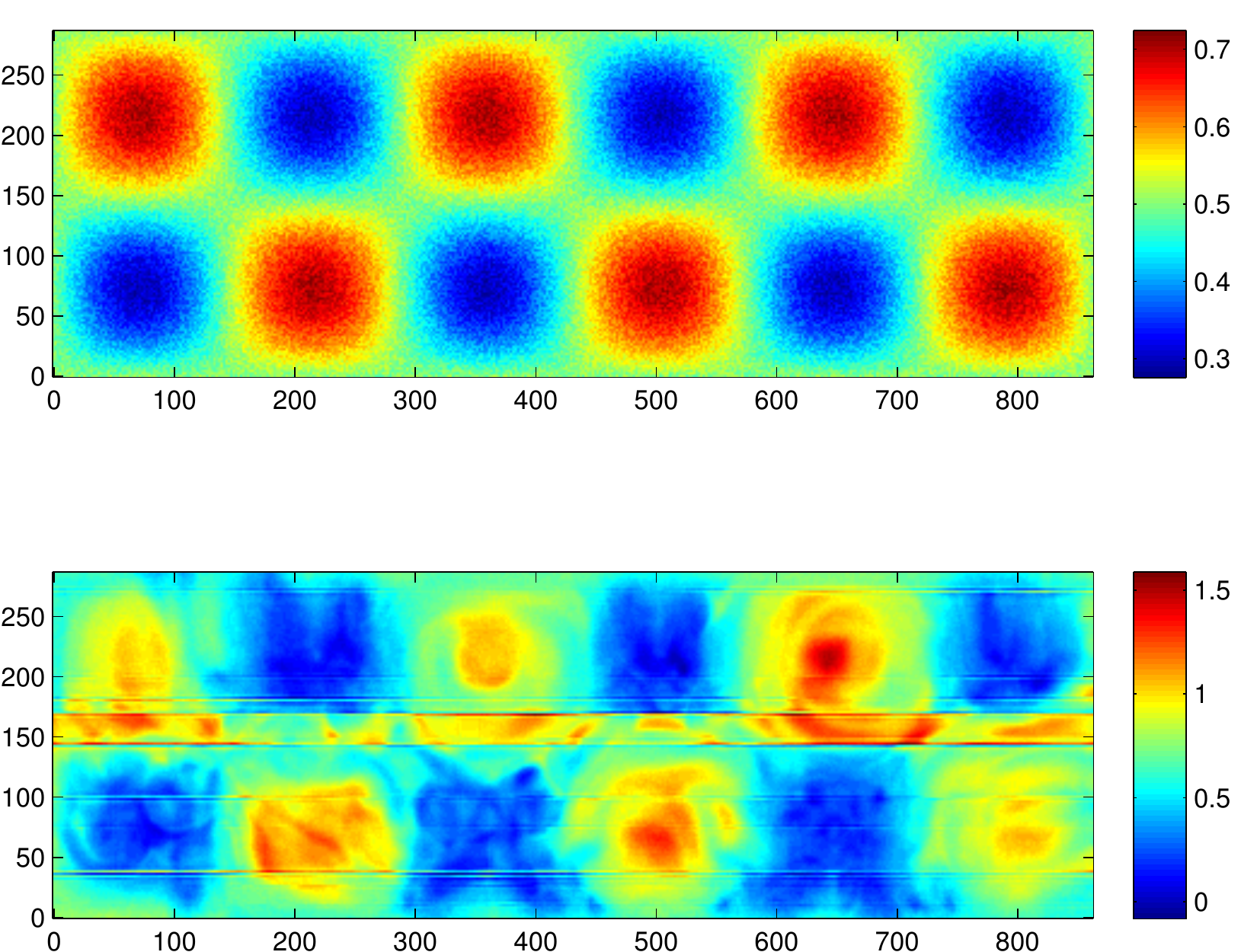}\label{fig:Autosync_m_clouds_1000}}\\[-.4cm]
\subfloat[][]{\includegraphics[width=.25\textwidth]{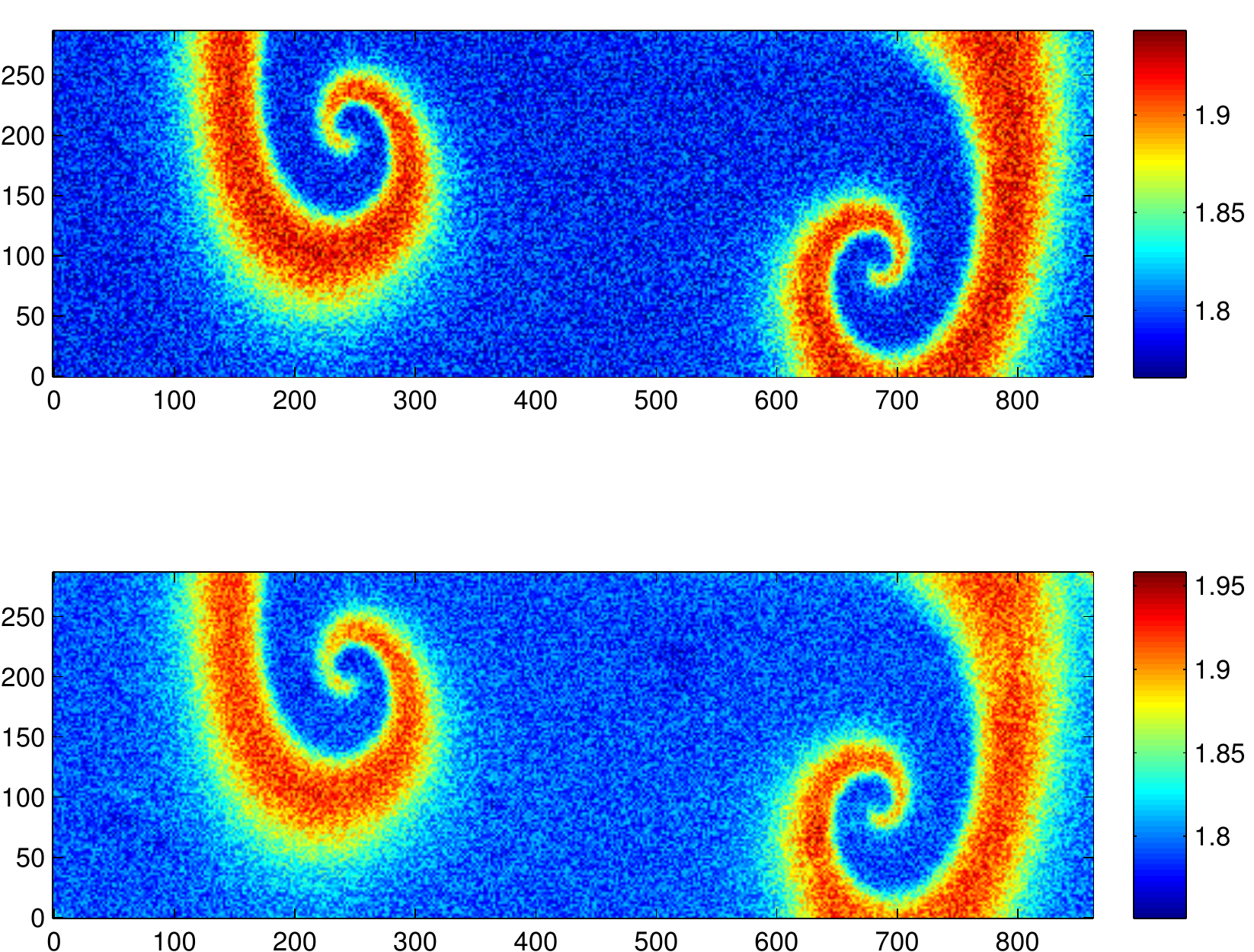}\label{fig:Autosync_k_clouds_FC}}
\subfloat[][]{\includegraphics[width=.25\textwidth]{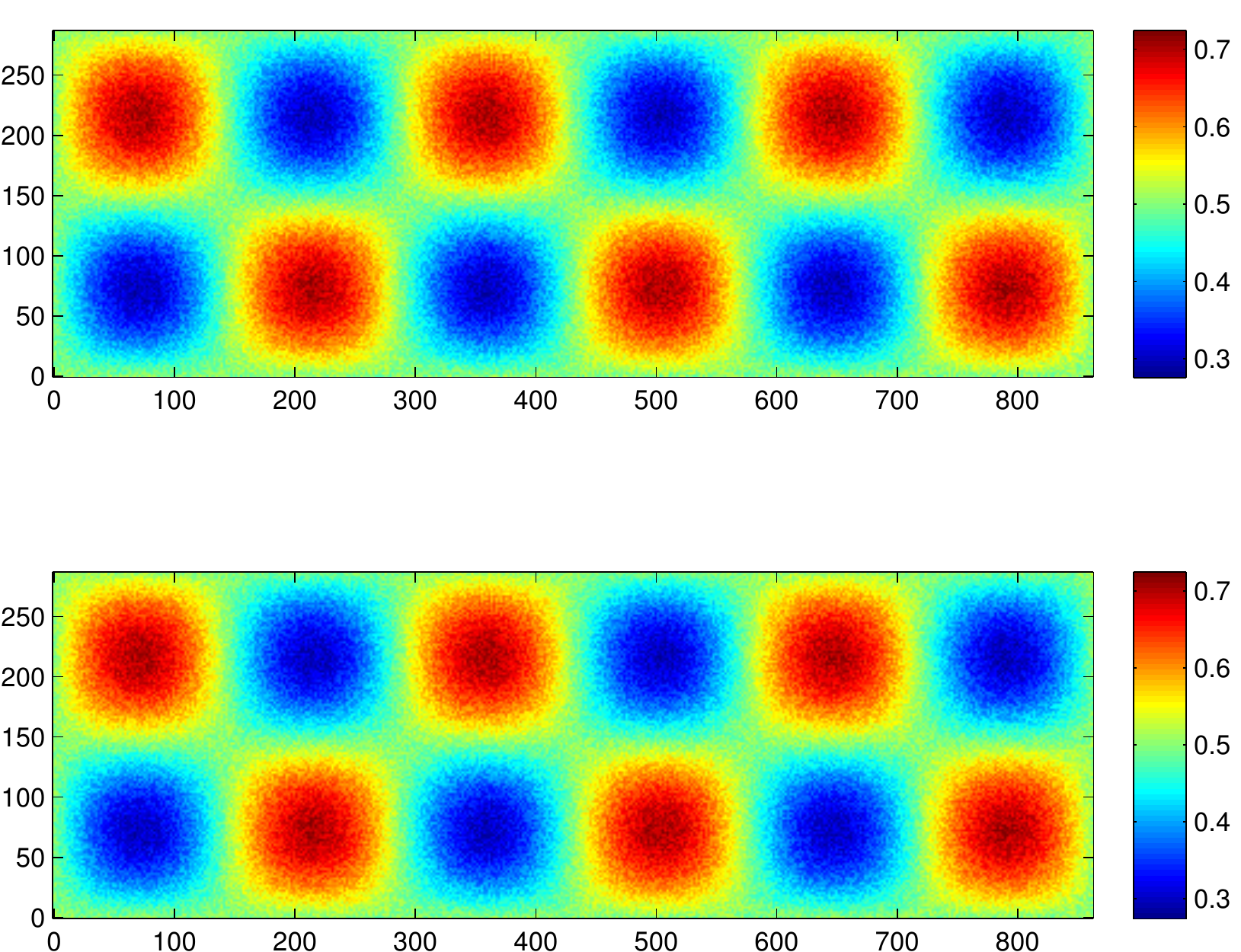}\label{fig:Autosync_m_clouds_FC}}
\caption{Autosynchronization of parameters with $25.5 \%$ of $\Omega$ hidden at any point in time from clouds and random noise added, however autosynchronization is observed. Each figure shows drive (top) and response (bottom) pairs. $k(x,y)$ and $\hat{k}(x,y,0)$ in \ref{fig:Autosync_k_clouds_IC}, $k(x,y)$ and $\hat{k}(x,y,200)$ in \ref{fig:Autosync_k_clouds_1000}, and $k(x,y)$ and $\hat{k}(x,y,8563)$ in \ref{fig:Autosync_k_clouds_FC}. $m(x,y)$ and $\hat{m}(x,y,0)$ in \ref{fig:Autosync_m_clouds_IC}, $m(x,y)$ and $\hat{m}(x,y,200)$ in \ref{fig:Autosync_m_clouds_1000}, and $m(x,y)$ and $\hat{m}(x,y,8563)$ in \ref{fig:Autosync_m_clouds_FC}.}
\label{fig:Clouds_Auto_Parameters}
\end{figure}
\begin{figure} \hspace*{-.1cm}
\subfloat[][]{\includegraphics[width=.25\textwidth]{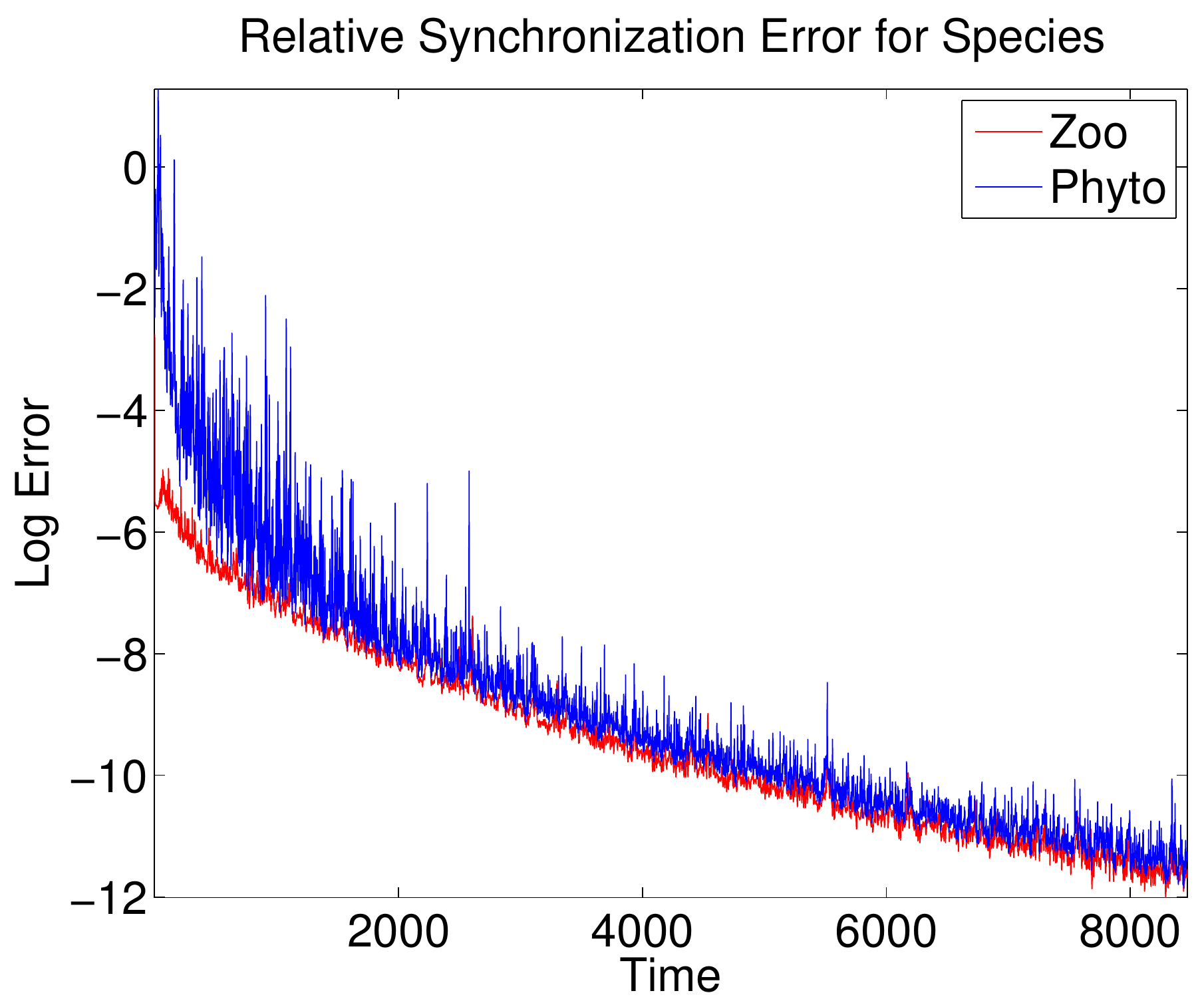}\label{fig:Autosync_PZ_Error_Clouds}} \hspace*{-.1cm}
\subfloat[][]{\includegraphics[width=.25\textwidth]{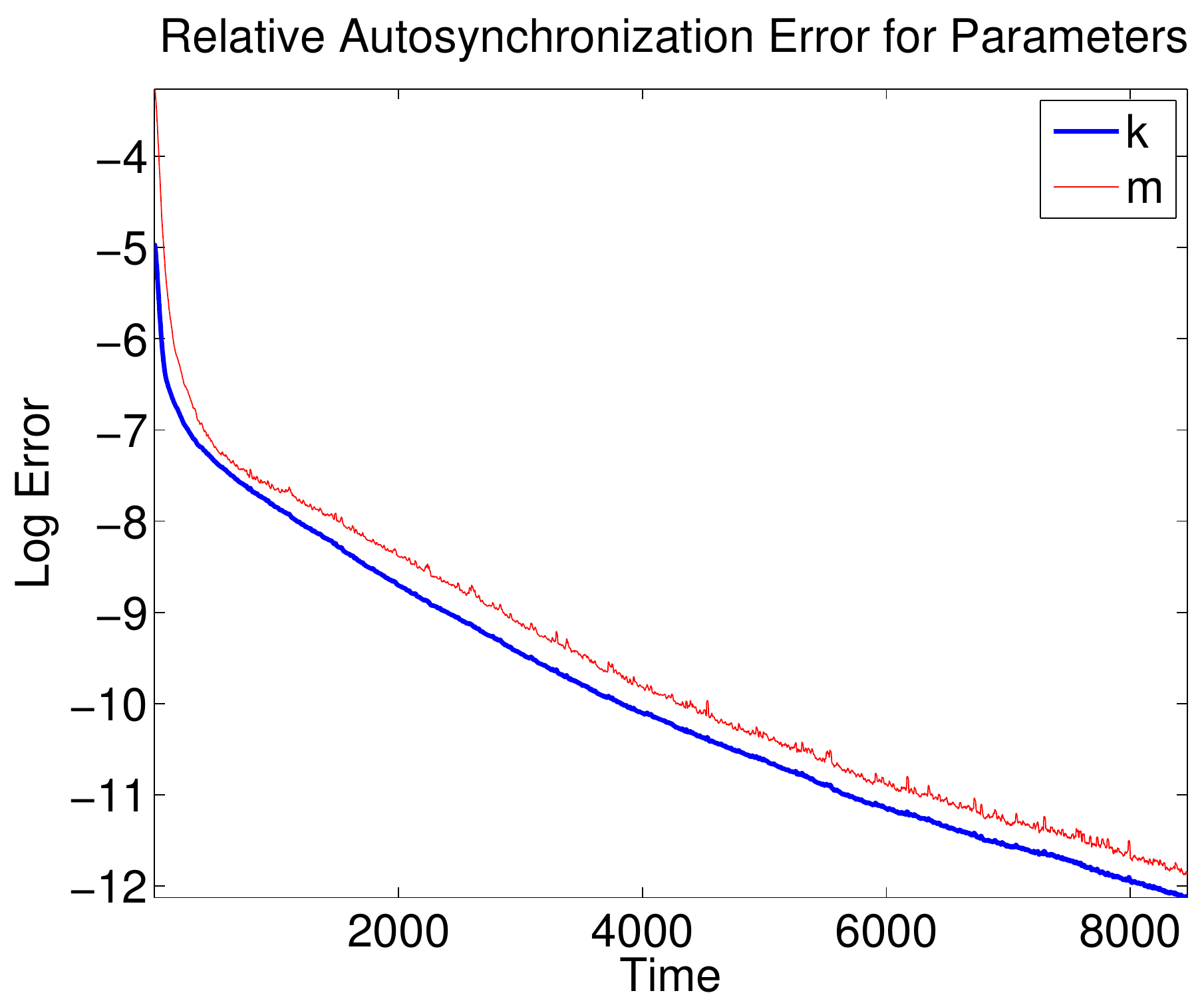}\label{fig:Autosync_km_Error_Clouds}}

\caption{Globally-averaged relative synchronization errors. Errors from simulation shown in (a) correspond to Figure \ref{fig:Clouds_Auto_Species} and (b) correspond to Figure \ref{fig:Clouds_Auto_Parameters}, shown to drop to within $1.2 \times 10^{-5}$ despite ever-present clouds.}
\label{fig:Clouds_Autosync_Error} 
\end{figure}

Figures \ref{fig:Clouds_Auto_Species} and \ref{fig:Clouds_Auto_Parameters} demonstrate a comparison between drive and response models. In the top of Figure \ref{fig:Autosync_phyto_clouds_IC}, we see the observed system $P(x,y,t)|_{\omega^C}$ wherein $25.5 \%$ of the data on $\Omega$ is not observable. Figures \ref{fig:Clouds_Auto_Species} and \ref{fig:Clouds_Auto_Parameters} demonstrate that phytoplankton, zooplankton, and both spatially dependent parameters $\hat{k}(x,y,t)$ and $\hat{m}(x,y,t)$ are estimated to high precision. 

Figure \ref{fig:Clouds_Autosync_Error} describes the globally-averaged relative error between the true system and the response system. The rate of convergence to the synchronization manifold is slower than with cloudless data as a result of allowing the systems to oscillate independently while not driven on $\omega$. For the simulation in Figures \ref{fig:Clouds_Auto_Species} and \ref{fig:Clouds_Auto_Parameters}, we choose $\kappa = 0.625$, $s_1 = 0.2$, and $s_2 = 0.6$ for good autosynchronization results. Summarizing, we have demonstrated that it is possible to fill in missing data when hidden by clouds and, as an added bonus, estimate noisy spatially-dependent model parameters with different functional forms. Similar results are obtained by testing other combinations of the two forms of model parameters in Figure \ref{fig:Parameters}.

Simulations are run for varying percentages of hidden data. Figure \ref{fig:errvshidden} shows the synchronization errors for simulations after a fixed time epoch of $t = 2400$ for all simulations. Specifically, the globally-averaged relative error between drive and response systems is plotted against the percentage of hidden data. It is clear that the speed of convergence of assimilation slows with respect to  the degree of occlusion. A counter-intuitive side note is that the assimilation quality actually improves by hiding data through about $13 \%$ before worsening as a larger percentage of data is hidden. For these simulations, the same initial conditions are used throughout for consistency. 

\begin{figure}[!h]
\hspace*{-.2cm}
\subfloat[][]{\includegraphics[width=.27 \textwidth]{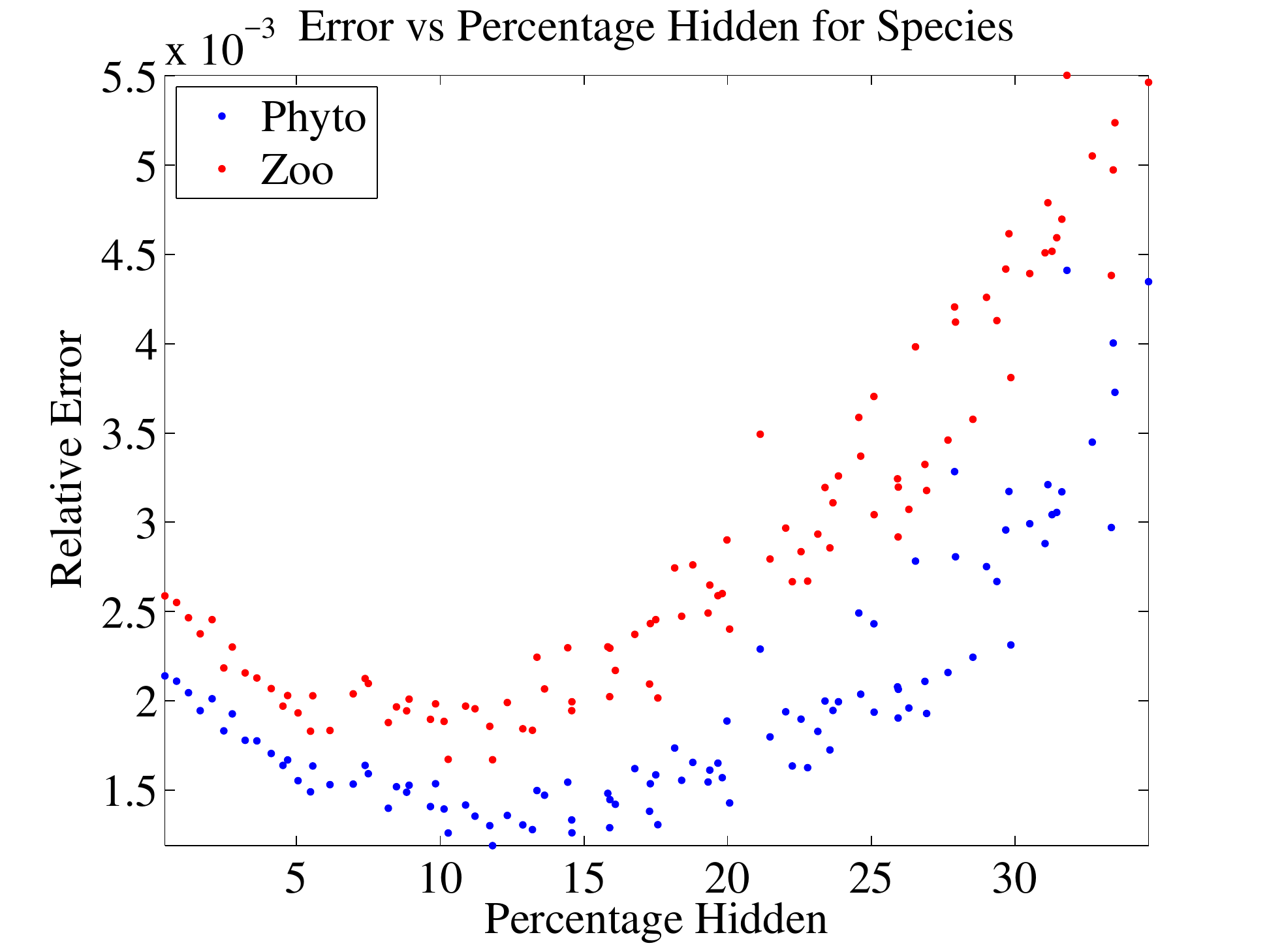}\label{fig:Hidden_vs_Error_Species}}  \hspace*{-.45cm}
\subfloat[][]{\includegraphics[width=.27 \textwidth]{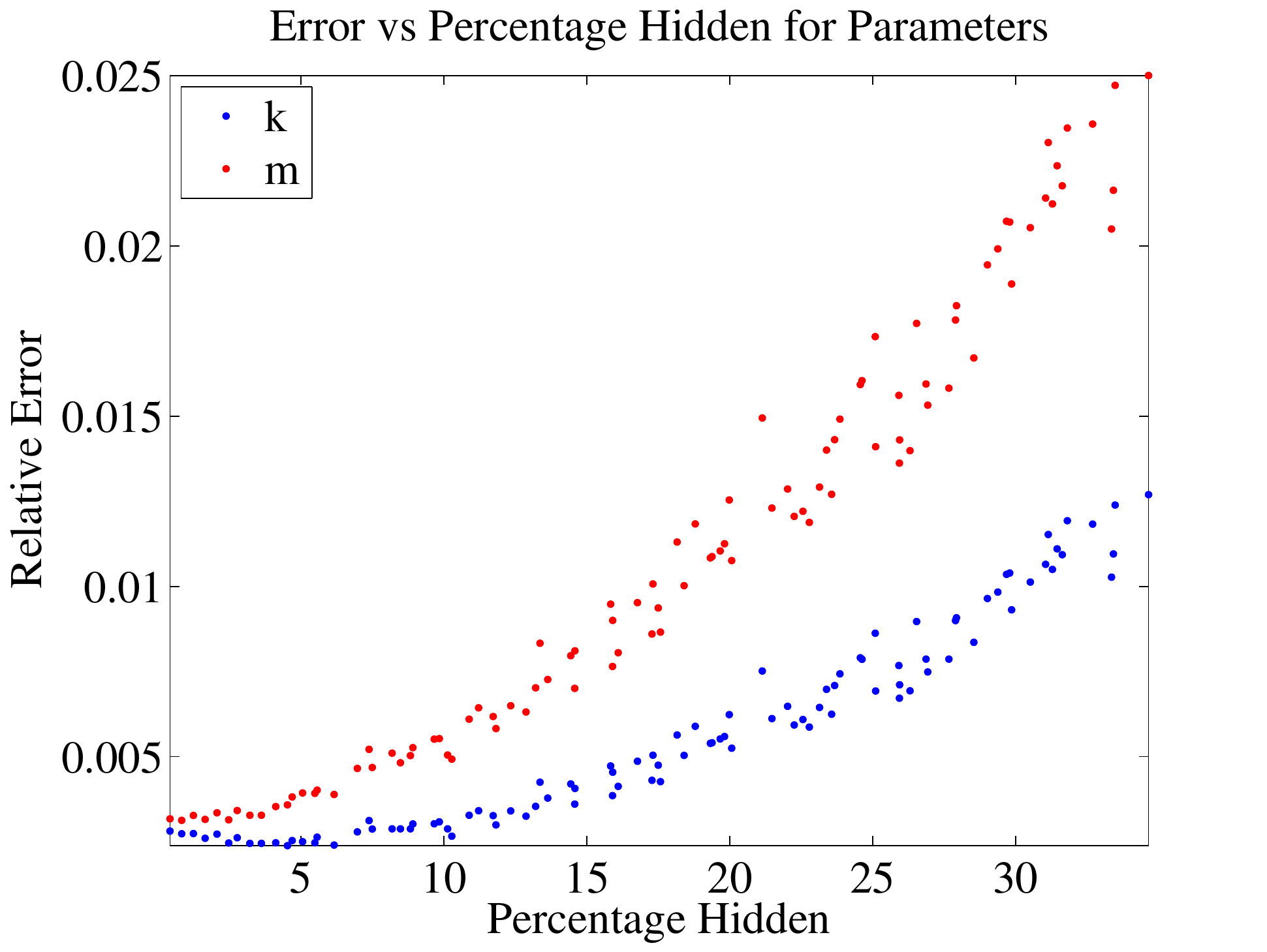}\label{fig:Hidden_vs_Error_Parameters}}
\caption{Synchronization error plotted against percentage of data hidden after simulation for $t = 2400$. Species shown in Figure \ref{fig:Hidden_vs_Error_Species} and parameters in Figure \ref{fig:Hidden_vs_Error_Parameters}.}
\label{fig:errvshidden}
\end{figure} 


We acknowledge inherent noise in remote sensing data and next demonstrate the method can work with fairly noisy observations. We add random noise to $P(x,y,t)$ for the length of the simulation, occlude $P(x,y,t)$ with $25.5 \%$ cloud coverage, and consider the same noisy, mixed functional form model parameters as above. As expected, as more noise is added to observations, the rate of synchronization slows and the error after a fixed epoch increases. Figure \ref{fig:errvsnoise} includes numerical experiments in which increasing noise is added to $P(x,y,t)$ and synchronization errors are compared over a fixed time epoch of $t = 2400$. 

\begin{figure}[!h]
\hspace*{-.2cm}
\subfloat[][]{\includegraphics[width=.27 \textwidth]{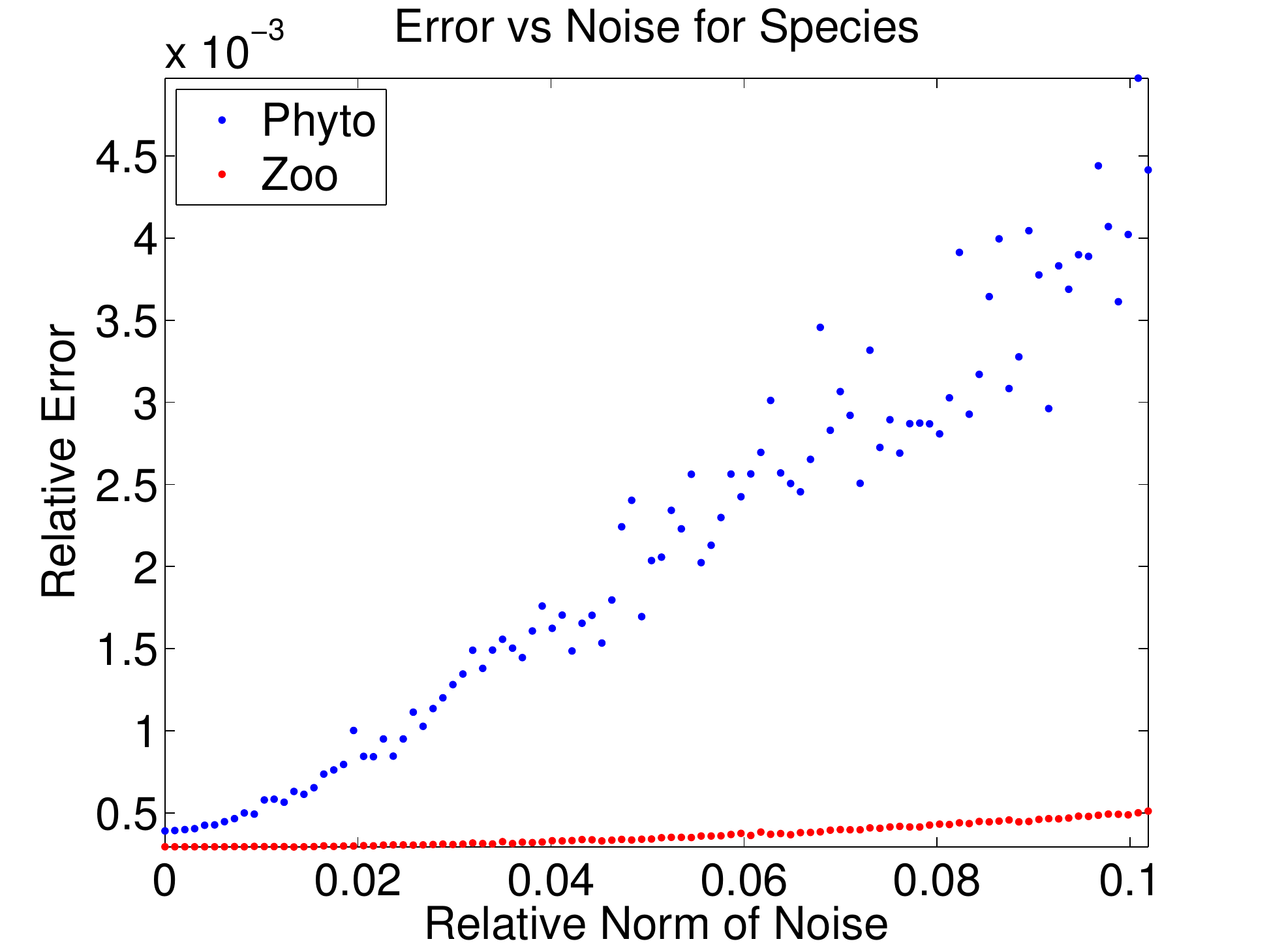}\label{fig:Noise_vs_Error_Species}}  \hspace*{-.45cm}
\subfloat[][]{\includegraphics[width=.27 \textwidth]{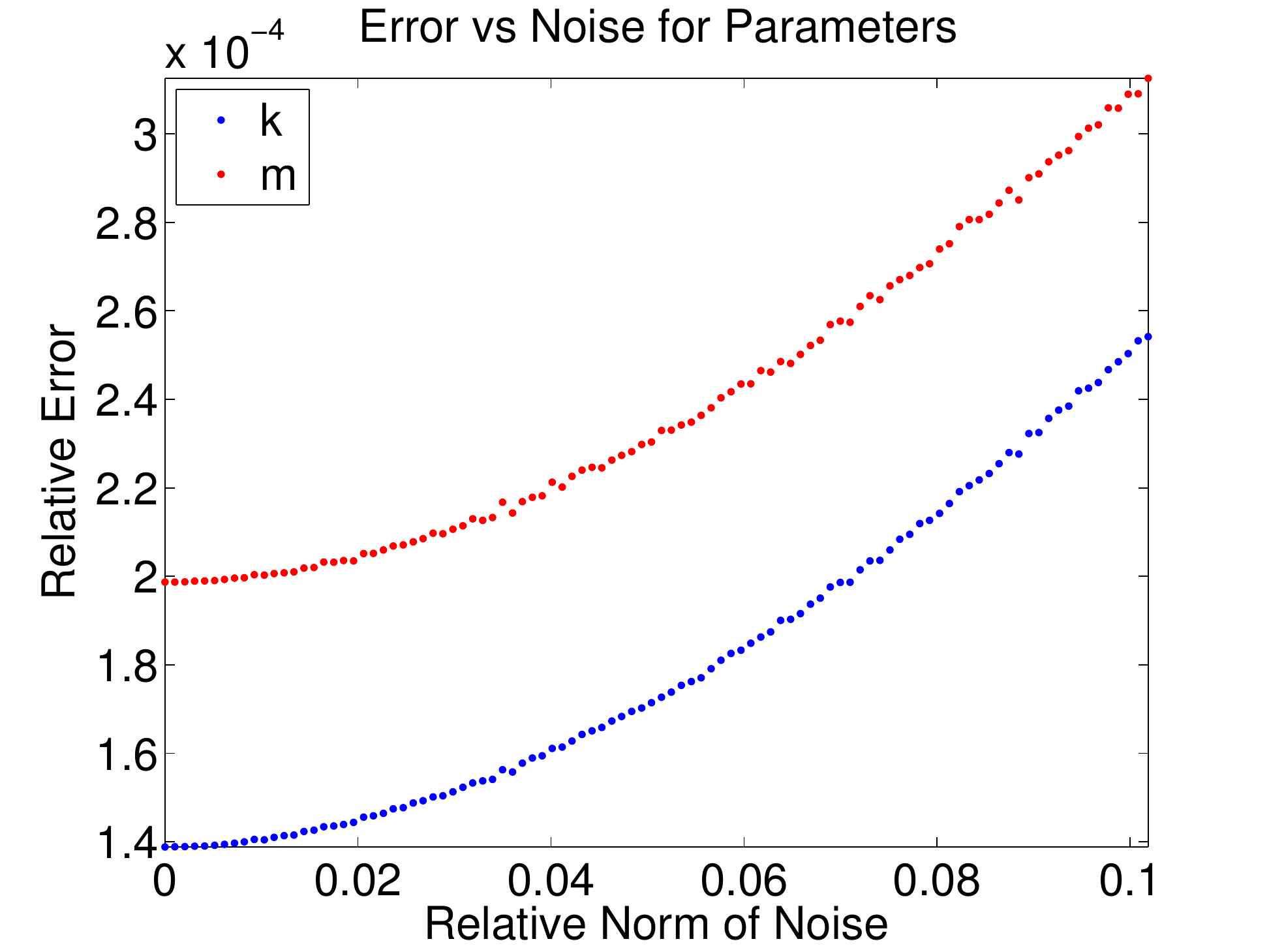}\label{fig:Noise_vs_Error_Parameters}}
\caption{Synchronization error plotted against amount of noise added to observations after simulation for $t = 2400$. Noise is normalized relative to amplitude of drive dynamics. Species shown in Figure \ref{fig:Noise_vs_Error_Species} and parameters in Figure \ref{fig:Noise_vs_Error_Parameters}.}
\label{fig:errvsnoise}
\end{figure}

To be considered practical, the method should provide decent results if data are available on a courser grid than that on which the model is evolved.

\section{Samples on a Coarse Grid}
Model simulations require data on an appropriately resolved grid, which may not line up with the grid on which samples are available. Often is the case that sampled data exist on a coarsened grid relative to a required simulation grid. We demonstrate that a simple modification of the technique will produce results similar to those above. We imagine that the domain is sampled in discrete patches, denoted by $S_n$, and on the patches we have available only local averages of true data as depicted in Figure \ref{fig:Coarser}. In this way, we sample a coarser subset of the domain and take local averages to be the new driving signal.

\begin{figure}[!h]
\hspace*{-.2cm}
\includegraphics[width=.5 \textwidth]{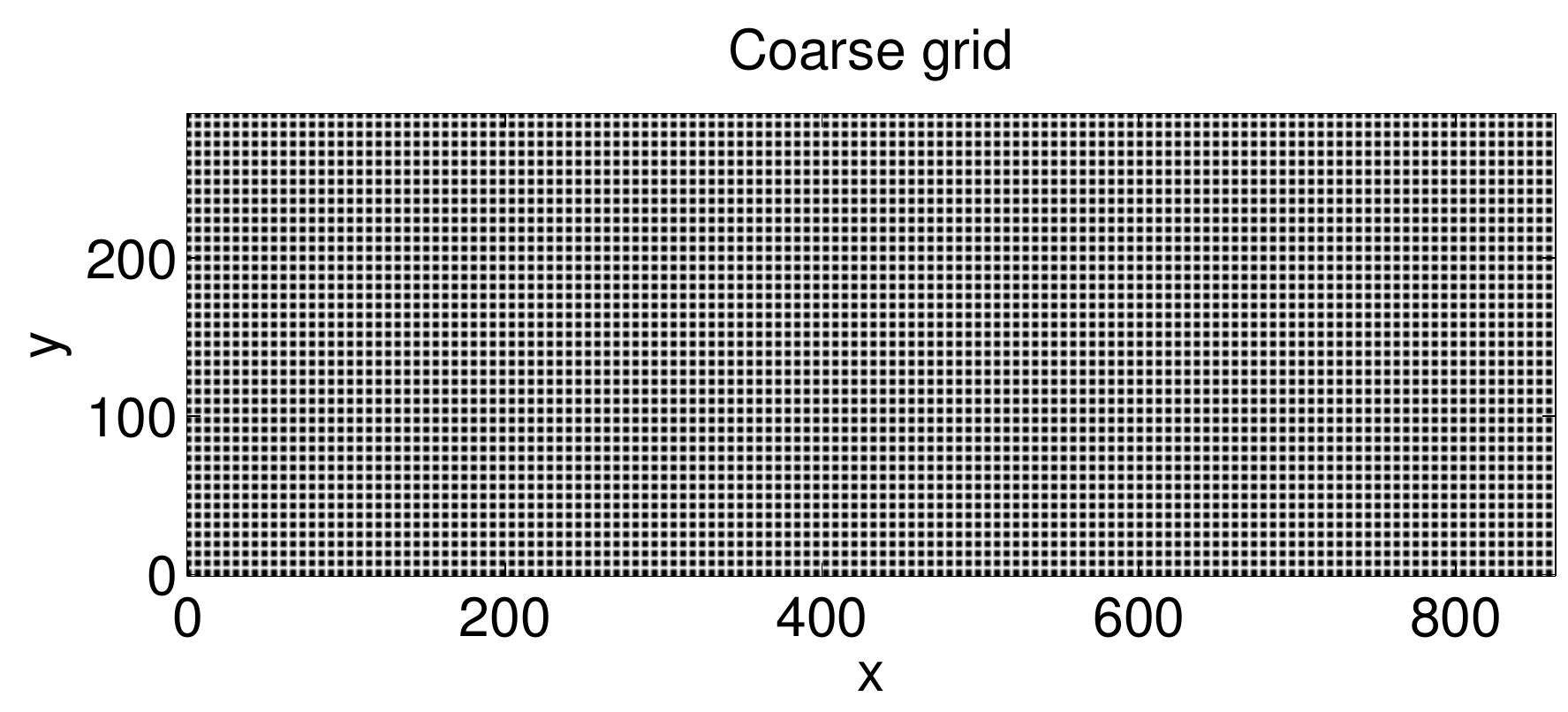}\vspace*{0cm}
\caption{Coarsely sampled domain with 2-pixel by 2-pixels sensors on which locally averaged data are sampled, and with 1 pixel between sensors wherein no data are available.}
\label{fig:Coarser}
\end{figure}

To adapt the problem to such data, we modify the response system, Eq \eqref{eq:Fish_Auto_Phyto_Clouds}

\begin{eqnarray}\label{eq:FishPhyto_Subsample}
\nonumber \frac{\partial \hat{P}}{\partial t} &=& \triangle \hat{P} + \hat{P}(1-\hat{P}) - \frac{\hat{P} \hat{Z}}{\hat{P} + h} + \kappa G_n \ \ \ \ \ \forall x,y \in S_n,\\
 \frac{\partial\hat{Z}}{\partial t} &=& \triangle \hat{Z} + \hat{k}\frac{\hat{P} \hat{Z}}{\tilde{P}+h} - \hat{m}\hat{Z},\\
\nonumber \frac{\partial \hat{k}}{\partial t} &=& \triangle \hat{k} + s_1(\tilde{P} - \hat{P}), \\
\nonumber \frac{\partial\hat{m}}{\partial t} &=& \triangle \hat{m} + s_2(\tilde{P} - \hat{P})\hat{P},
\end{eqnarray}

where $\tilde{P}$ represents locally averaged observations from the drive system and 

\begin{equation}
G_n(t) = \frac{1}{(dx)(dy)} \sum_{x,y \in S_n}(P(x,y,t) - \hat{P}(x,y,t)),
\end{equation}
where $S_n$ is the rectangular ``sensor'' on which local averaging occurs. 

Note, for best results, diffusion is added to the parameter equations in Eq \eqref{eq:FishPhyto_Subsample} in order that data from the driven regions, $S_n$ will diffuse into the occluded regions. To further mimic our target problem of remote sensing data, we add white Gaussian noise to the observations. We now demonstrate the method reconstructs parameters and unknown states on incomplete, noisy, patchy experimental data. 

Figure \ref{fig:NoiseSub} includes the results of simulations with coupling given by Eq \eqref{eq:FishPhyto_Subsample}, with a sensor size of 2 pixels by 2 pixels, and a spacing of 1 pixel between sensors. Here we show the end of the simulation for brevity. It is obvious the method suffers from noise, local averaging, and missing data between sensors. In Figure \ref{fig:Noise_Sub_P_FC}, it is clear a considerable amount of data is occluded, the data that are available include noise, and still unknown states and parameters are reconstructed fairly well.

\begin{figure}
\centering
\subfloat[][]{\includegraphics[width=.25\textwidth]{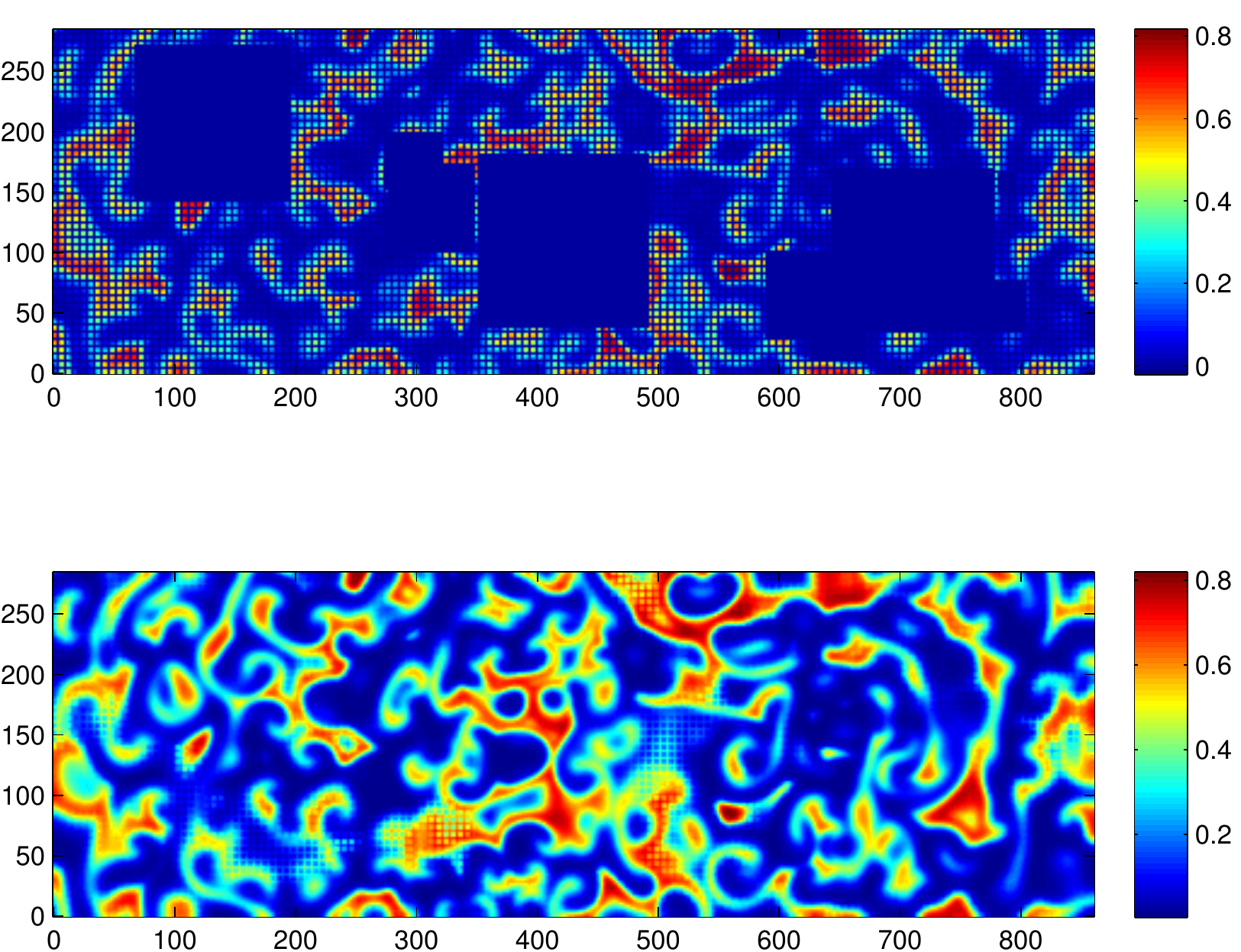}\label{fig:Noise_Sub_P_FC}}
\subfloat[][]{\includegraphics[width=.25\textwidth]{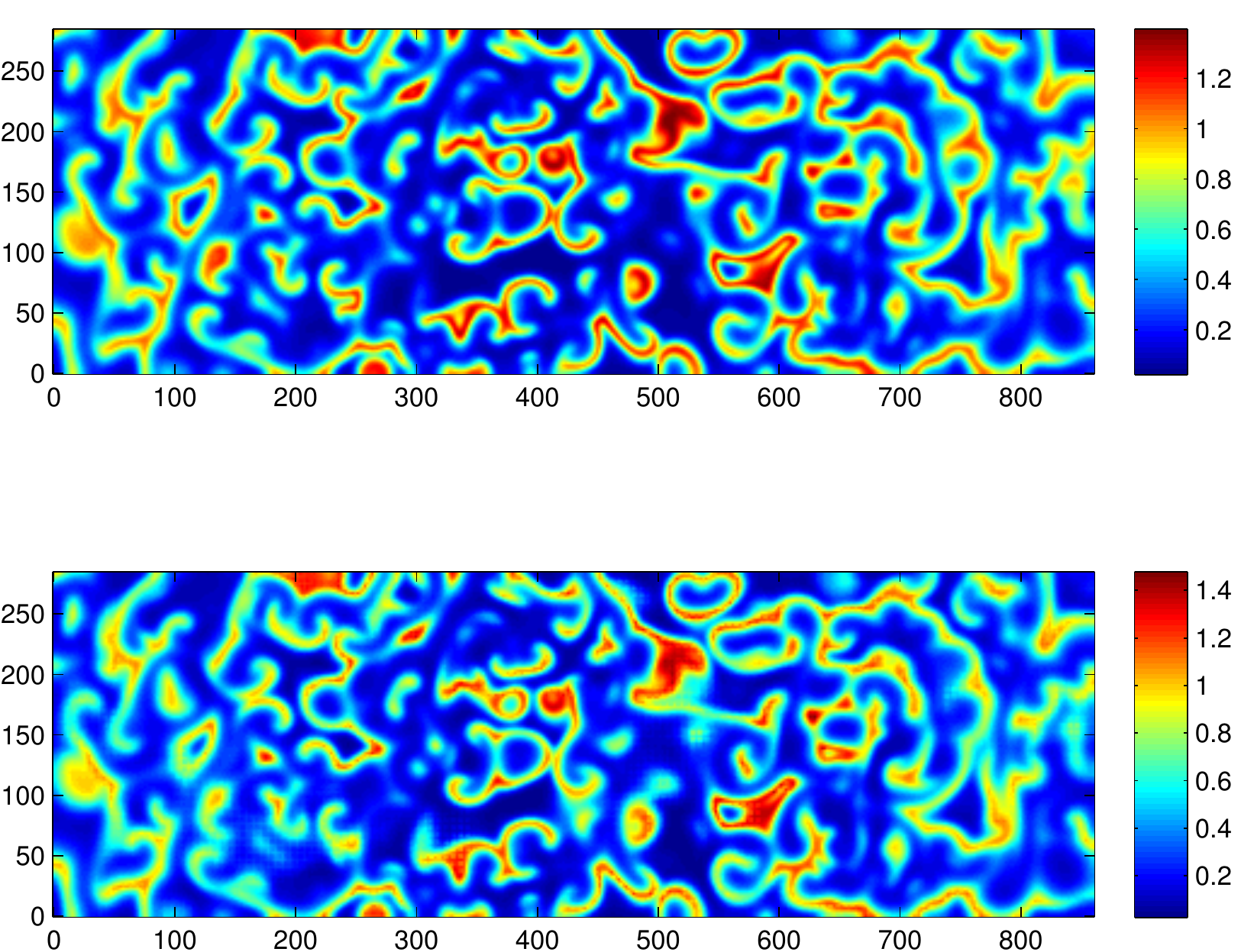}\label{fig:Noise_Sub_Z_FC}}\\[-.4cm]
\subfloat[][]{\includegraphics[width=.25\textwidth]{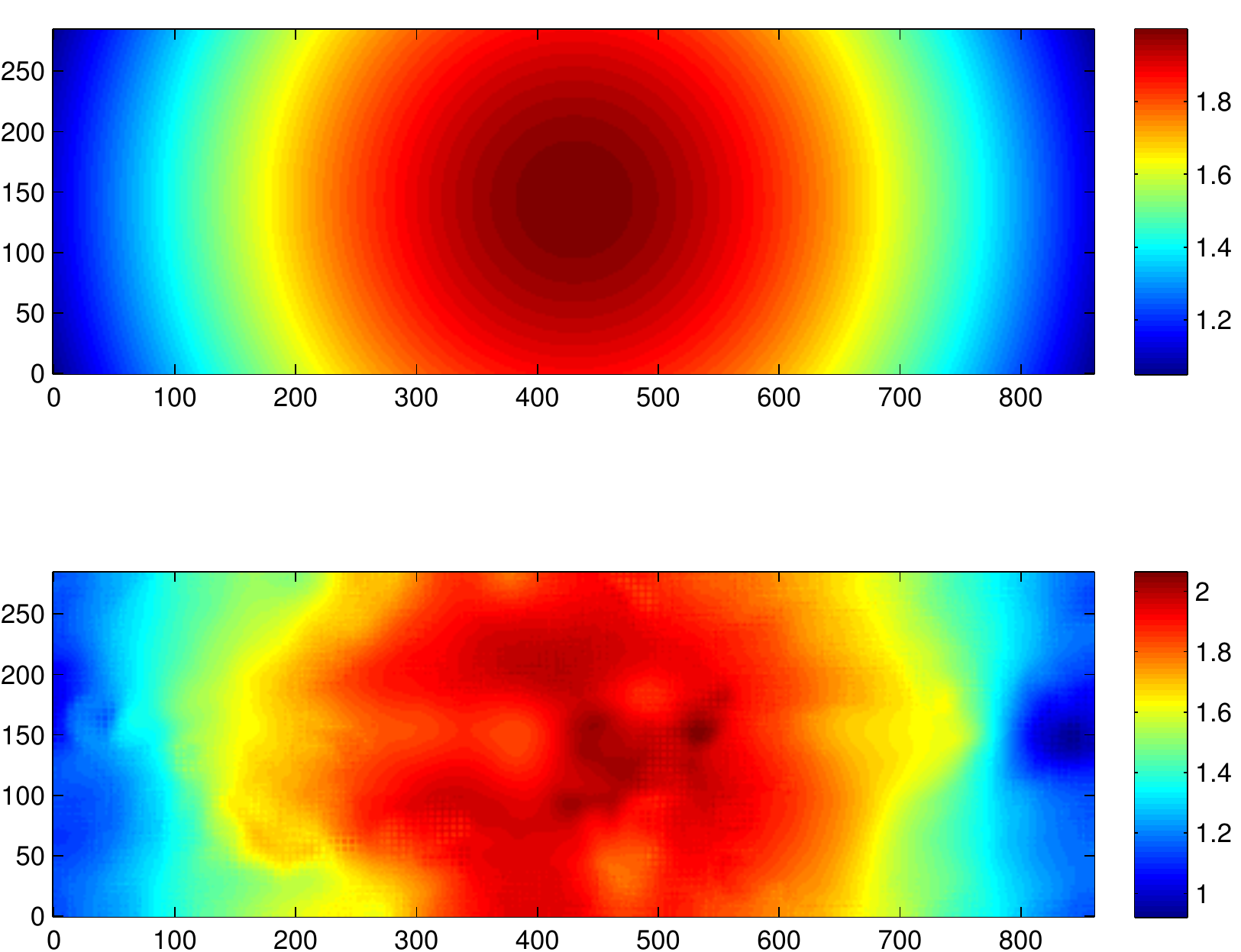}\label{fig:Noise_Sub_k_FC}}
\subfloat[][]{\includegraphics[width=.25\textwidth]{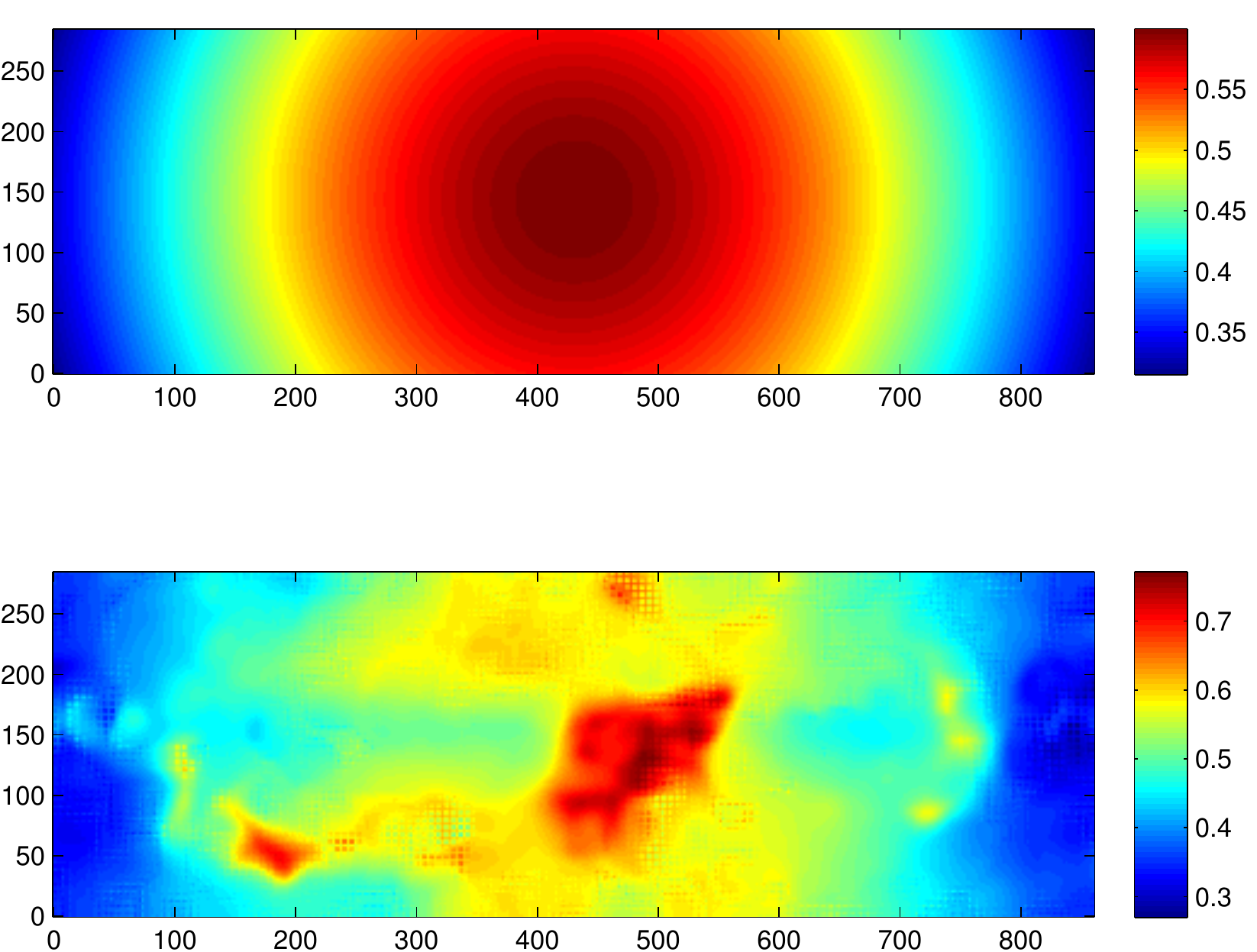}\label{fig:Noise_Sub_m_FC}}\\[-.4cm]
\caption{Autosynchronization of states and parameters with $25.5 \%$ of $\Omega$ hidden at any point in time from clouds, data available on a course grid, and random noise added. The messy available data is evident at the top of \ref{fig:Noise_Sub_P_FC}. Each figure shows drive (top) and response (bottom) pairs. $P(x,y,4000)$ and $\hat{P}(x,y,4000)$ in \ref{fig:Noise_Sub_P_FC}, $Z(x,y,4000)$ and $\hat{Z}(x,y,4000)$ in \ref{fig:Noise_Sub_Z_FC}, and $k(x,y)$ and $\hat{k}(x,y,4000)$ in \ref{fig:Noise_Sub_k_FC}. $m(x,y)$ and $\hat{m}(x,y,4000)$ in \ref{fig:Noise_Sub_m_FC}.}
\label{fig:NoiseSub}
\end{figure}

In Figure \ref{fig:Err_vs_Space} we provide simulations to demonstrate the dependency of synchronization quality on the sensor spacing. The synchronization errors for zooplankton and the parameter $k(x,y)$ are plotted against time for simulations admitting zero, one, and two pixels between sampling sensors on the domain. If sensor spacing is too sparse, the method struggles to fill in data between patches on which local averages are provided, and ultimately fails. We note that the parameter $m(x,y)$ and phytoplankton exhibit similar behavior. 

\begin{figure}[!h]
\hspace*{-.2cm}
\subfloat[][]{\includegraphics[width=.25 \textwidth]{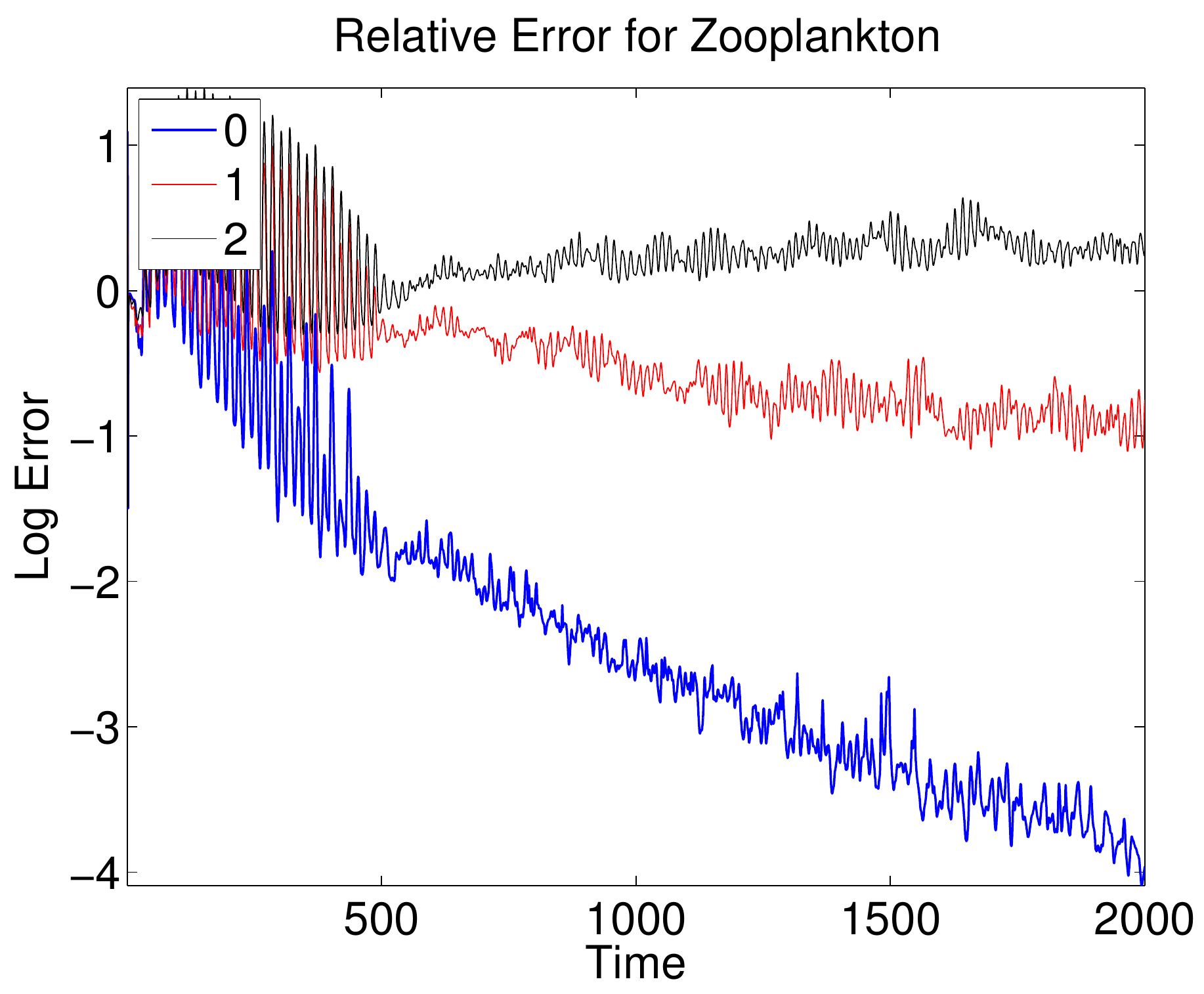}\label{fig:Compare_Spacing_Error_Species}}  \hspace*{-.1cm}
\subfloat[][]{\includegraphics[width=.25 \textwidth]{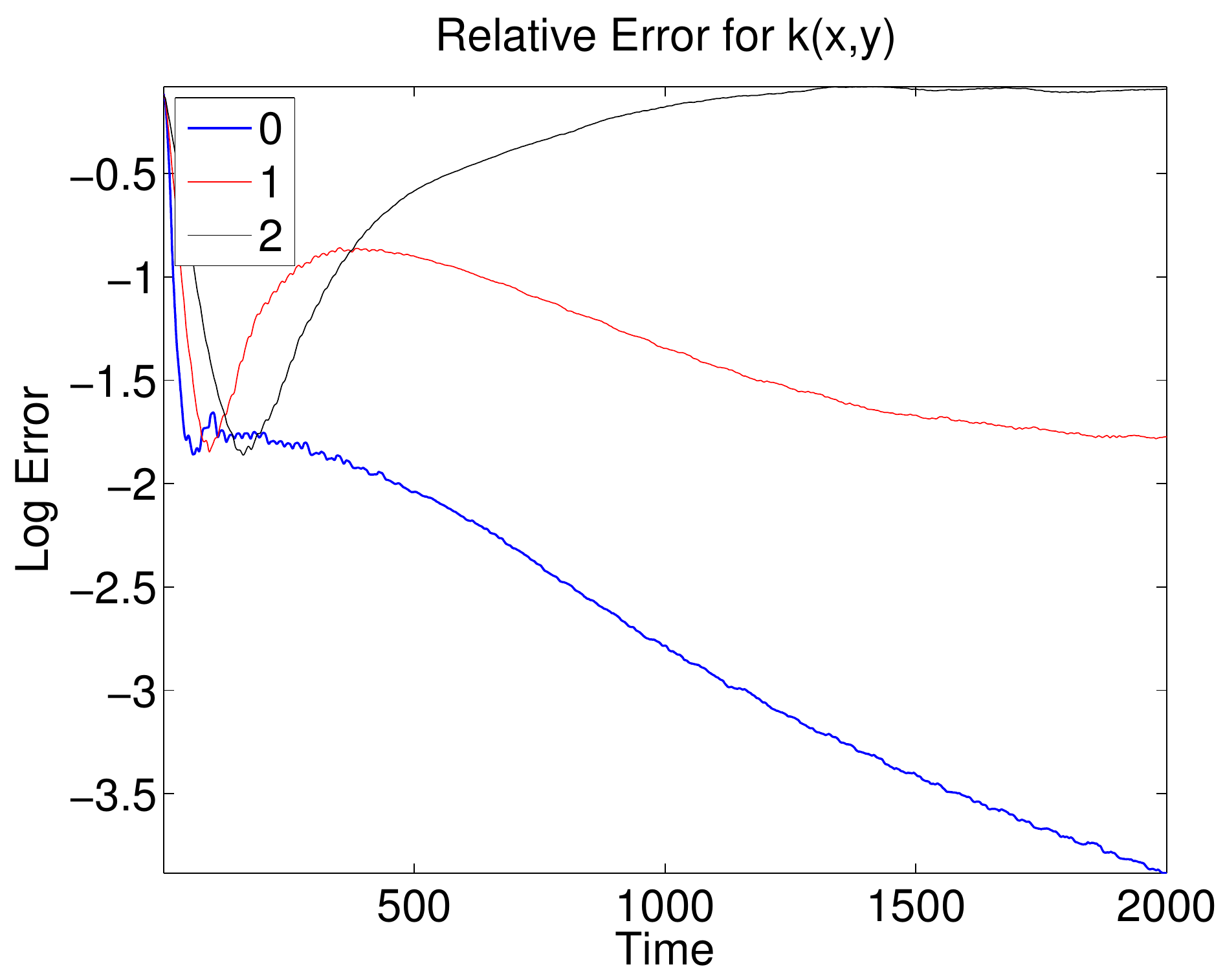}\label{fig:Compare_Spacing_Error_Parameters}}
\caption{Synchronization error over time for different amounts of spacing between locally averaged data. Species shown in Figure \ref{fig:Compare_Spacing_Error_Species} and parameters in Figure \ref{fig:Compare_Spacing_Error_Parameters}.}
\label{fig:Err_vs_Space}
\end{figure} 

These results demonstrate potential for future application to real data. In fact the autosynchronization method is capable of data assimilation by revealing hidden phytoplankton and zooplankton densities and estimating model parameters on noisy, coarse, cloudy data. Furthermore, our simulations indicate that state estimation by synchronization is more robust to coarse data than autosynchronization, which may be of use if parameter estimates are not required. This is merely a demonstration of autosynchronization applied to a remote sensing problem. These results require analytical reinforcement, including a discussion of the basin of attraction for the synchronization manifold and allowable coupling strengths to observe synchronization and parameter estimation. In the following, we provide analysis of manifold stability.

\section{Analysis}

To better understand this tendency to synchronize despite hidden data, we are inspired by a method from network theory. We represent the system as a moving neighborhood network and define each pixel in the image domain $\Omega$, to be an individual oscillator, $\mathbf{u_j}$. We include the drive and response images in the network so that an $m \times n$ image provides $2N = 2mn$ oscillators.  

Therefore, each drive oscillator $\mathbf{u_j}$, for some $j = 1:N$, feeds a response oscillator corresponding to the same spatial pixel, $\mathbf{\hat{u}_{j}}$, for some $j=N+1:2N$. The drive system is hidden for a time epoch $\Delta$ to represent intermittent cloud cover. Pixels over which there are clouds are uncoupled while covered. In time the network topology shifts thereby shifting the coupling between drive and response oscillators over $\Omega$. 

The adjacency matrix of our directed weighted random graph has zeros on the main diagonal and is defined as: 
 \begin{align}
\nonumber a_{ij} = \begin{cases}
\ \ w_{ij}  \mathrm{\  with \ probability \ } p_{ij}\\
\ \ 0 \mathrm{ \ with \ probability \ }1 - p_{ij},
\end{cases}
\end{align}
 for $i \neq j$. We subtract the adjacency matrix from the matrix $D$, with nonzero elements on the main diagonal $d_i = \sum_{j = 1}^n a_{ij}$, to obtain the graph Laplacian $L$. Our network requires edges between nearby neighbors for diffusion, and between images for both direct replacement and diffusive drive-response coupling. We choose to order the 2$N$ vertices representing image pixels from left to right, taking a row at a time, and placing the drive image first, followed by the response image. That is we stack subsequent rows of the drive image followed by the response image to build the vector of 2$N$ components. 
 
Here we analyze the occluded synchronization system \eqref{eq:Fish_Sync_Phyto_clouds}. We define two Laplacians to represent diffusion and drive-response coupling. \eq{L_1 \in \mathbb{R}^{2N \times 2N}} is a sparse matrix with weighted entries corresponding to a chosen diffusion stencil and boundary conditions. \eq{L_2(t) \in \mathbb{R}^{2N \times 2N}}, with elements $l^2_{ij}(t)$ is a time-dependent sparse matrix that represents diffusive coupling with a switching network topology and is fixed in the time interval $T_k = [t_k,t_{k+1}) $
 
 \begin{align}
\nonumber  L_2(t) = 
\left[
\begin{array}{cc}
 \textbf{0} & \textbf{0}  \\
  I_{N} &  -I_{N}
\end{array}
\right] , \end{align} 
 
\noindent where \eq{\textbf{0}} is an $N \times N$ zero matrix. Next, we define two matrices to model coupling between species, $B_1 = I_{2  \times 2}$ and $B_2 =  [1 \ 0; 0 \ 0]$. The drive dynamics for the first $N$ oscillators and the response dynamics for the remaining $N$ oscillators are
\begin{eqnarray}\label{eq:stochastic}
 \mathbf{\dot{u}}_q (t) &=& \mathbf{f}(\mathbf{u}_{q}) + \sigma_1 B_1\sum^{2N}_{j=1}l^1_{qj}(t)\mathbf{u}_j,  \ \ \ \ \ \ q = 1:N, \\ 
\nonumber \mathbf{\dot{u}}_q (t) &=& \mathbf{f}(\mathbf{u}_{q}) + \sigma_1 B_1\sum^{2N}_{j=1}l^1_{qj}(t)\mathbf{u}_j + \sigma_2 B_2\sum^{N}_{j=1}l^2_{qj}(t)\mathbf{u}_j, \\  \nonumber && \hspace*{4.6cm} q = N+1:2N, 
\end{eqnarray}
 respectively, where $\mathbf{u}_q \in {\mathbb{R}}^2$ is the state of the $q$th oscillator, \eq{\sigma_1 = \frac{1}{h^2}}, \eq{\sigma_2 = \kappa}, and $\mathbf{f} : \mathbb{R}^2 \to \mathbb{R}^2 $ describes the individual dynamics of each oscillator.
 
 We linearize the system about the synchronization manifold as
\begin{align}\label{eq:linearized}
\dot{\delta u}(t) = (F + \sigma_1 L^1 \otimes B_1 + \sigma_2 L^2(t) \otimes B_2) \delta u(t),
\end{align}

where 
 \begin{align}
\nonumber  F &= I_{2N} \otimes J_i, \ \ \ \textrm{and} \end{align} 
 \begin{align}
\nonumber  J_i &= 
\left[
\begin{array}{cc}
\left(1 - 2P_i - \frac{h}{(P_i + h)^2}\right)  & -\frac{P_i}{P_i + h} \\
  \frac{h k}{(P_i + h)^2} &  \frac{k P_i}{P_i + h} - m
\end{array}
\right] . \end{align} 

We decompose \eqref{eq:linearized} into a component that evolves along the synchronization manifold and a transverse component with a matrix \eq{W \in \mathbb{R}^{2N \times (2N-1)}}. The state vector \eq{\delta u (t)} is decomposed, with $e$ denoting the standard basis vector in \eq{\mathbb{R}^2}, as
\begin{align}
\nonumber \delta u (t) = (W \otimes I_2) \zeta(t) + e \otimes \delta u_s (t),
\end{align}
\noindent where
\begin{align}
\nonumber \zeta(t) = (W \otimes I_2)^T \delta u(t), 
\end{align}
\noindent and
\begin{align}
\nonumber \delta u_s(t) = \frac{1}{N} ((e \otimes I_2)^T \delta u (t)).
\end{align}

The linearized dynamics \eqref{eq:linearized} are partitioned as 
\begin{eqnarray}
\nonumber \dot{\delta u_s}(t) &=& F(t) \delta u_s(t) + \sigma_1(e^T L^1 W \otimes B_1)\zeta(t) \\ \nonumber && + \ \sigma_2(e^T L^2(t) W \otimes B_2)  \zeta(t), \ \ \ \ \textrm{and}\\[.2cm]
\nonumber \dot{\zeta}(t) &=& I_{2N-1} \otimes F(t) + \sigma_1 (W^T L^1 W \otimes B_1)\zeta(t) \\ \nonumber && + \ \sigma_2 (W^T L^2(t) W \otimes B_2) \zeta(t),
\end{eqnarray}

where almost sure asymptotic synchronization of \eqref{eq:stochastic} is observed if \eq{\zeta(t)} almost surely converges to zero \cite{kushner1967stochastic}. 
 
It has been proven \cite{porfiri2008synchronization} that if the un-occluded system is uniformly asymptotically stable, then so is the occluded system, provided that the system is observable often enough. For completeness we restate Theorem 1 found in \cite{porfiri2008synchronization}:

\begin{thm} \label{thm1} Consider the deterministic dynamic system:
\begin{align}\label{eq:network}
\dot{y}(t) = (I_{N-1} \otimes F(t) + \sigma W^TE[L]W \otimes B(t))y(t),
\end{align}
representing the linearized transverse dynamics of 
\begin{align}\label{eq:networktwo}
\dot{x}_q(t) = f(x_q(t)) + \sigma B(t) \sum^N_{j=1} E[l_{qj}(t)]x_j(t),\\
\nonumber q= 1,...,N, t \in \mathbb{R}^+
\end{align}
where $E[L]$ is the time-averaged graph Laplacian.  Assume that F(t) and B(t) are bounded and continuous for all $t \geq 0$. If \eqref{eq:network} is uniformly asymptotically stable, there is a time-scale $\Delta^* > 0$ such that for any shorter time-scale $\Delta < \Delta^*$, the stochastic system \eqref{eq:stochastic} locally asymptotically synchronizes almost surely. 
\end{thm}

Now we check that the hypothesis of this theorem holds in the scenario of interest to us in this paper, which is that the deterministic graph Laplacian supports synchronization and the switching period $\Delta$ between network topologies is sufficiently small.
\begin{cor}\label{cor1}
Consider the stochastic system \eqref{eq:stochastic}, where $L_2(t)$ represents the switching network topology induced by moving clouds with speed $\nu$. Suppose \eq{\omega} is the set of unobservable data over the domain \eq{\Omega}. If \eqref{eq:network} is uniformly asymptotically stable, there exists a $C>0$ and $\nu^*(C)$, such that if $||\omega|| < C$ and $\nu > \nu^*(C)$, then the time-scale between switching network topologies is sufficiently small, $\Delta < \Delta^*$, and the stochastic system locally asymptotically synchronizes almost surely.

\end{cor}

\begin{proof}
The proof of Theorem \ref{thm1} found in \cite{porfiri2008synchronization} need only be altered slightly for our modification. 
We modify the definition of \eq{M(t)} so that 
\begin{eqnarray}
\nonumber M(t) &=& I_{2N-1} \otimes F(t) + \sigma_1 W^T L^1 W \otimes B_1 \\ \nonumber && + \ \sigma_2 W^T L^2(t) W \otimes B_2 , \ \ \ \ \ \textrm{and} \\[.2cm]
 \nonumber \bar{M}(t) &=& I_{2N-1} \otimes F(t) + \sigma_1 W^T E[L^1] W \otimes B_1 \\ \nonumber && + \ \sigma_2 W^T E[L^2] W \otimes B_2.
 \end{eqnarray}
We note that both $B_1$ and $B_2$ are bounded and continuous for $t > 0$, and there are positive constants \eq{\beta_1, \beta_2,} and \eq{\beta}, such that for any \eq{t \geq 0}, \eq{|| B_1|| \leq \beta_1, \textrm{ and } || B_2(t) || \leq \beta_2}. Then setting \eq{\beta = \mathrm{max}(\beta_1, \beta_2)}, the remainder of the proof follows. Therefore, we are guaranteed a time-scale, $\triangle^* > 0$, below which the system will asymptotically synchronize.
\end{proof}

The referenced time scale is the time epoch under which the drive system is occluded. That is, if the drive system is occluded for too long, corresponding to clouds that are too large or plentiful, the systems will not exhibit synchronization. 

To demonstrate Corollary \ref{cor1} with respect to our problem, we simulate with varied rates of cloud movement. As cloud movement slows, $\Delta$ increases, and the graph Laplacian switches less frequently. If cloud movement slows too much, then the systems will not exhibit synchronization. Figure \ref{fig:errvsspeed} includes the results of simulations over a fixed time epoch of $t = 2400$. Parameter and state errors increase as the rate of cloud movement decreases. 

\begin{figure}[!h]
\hspace*{-.3cm}
\subfloat[][]{\includegraphics[width=.255 \textwidth]{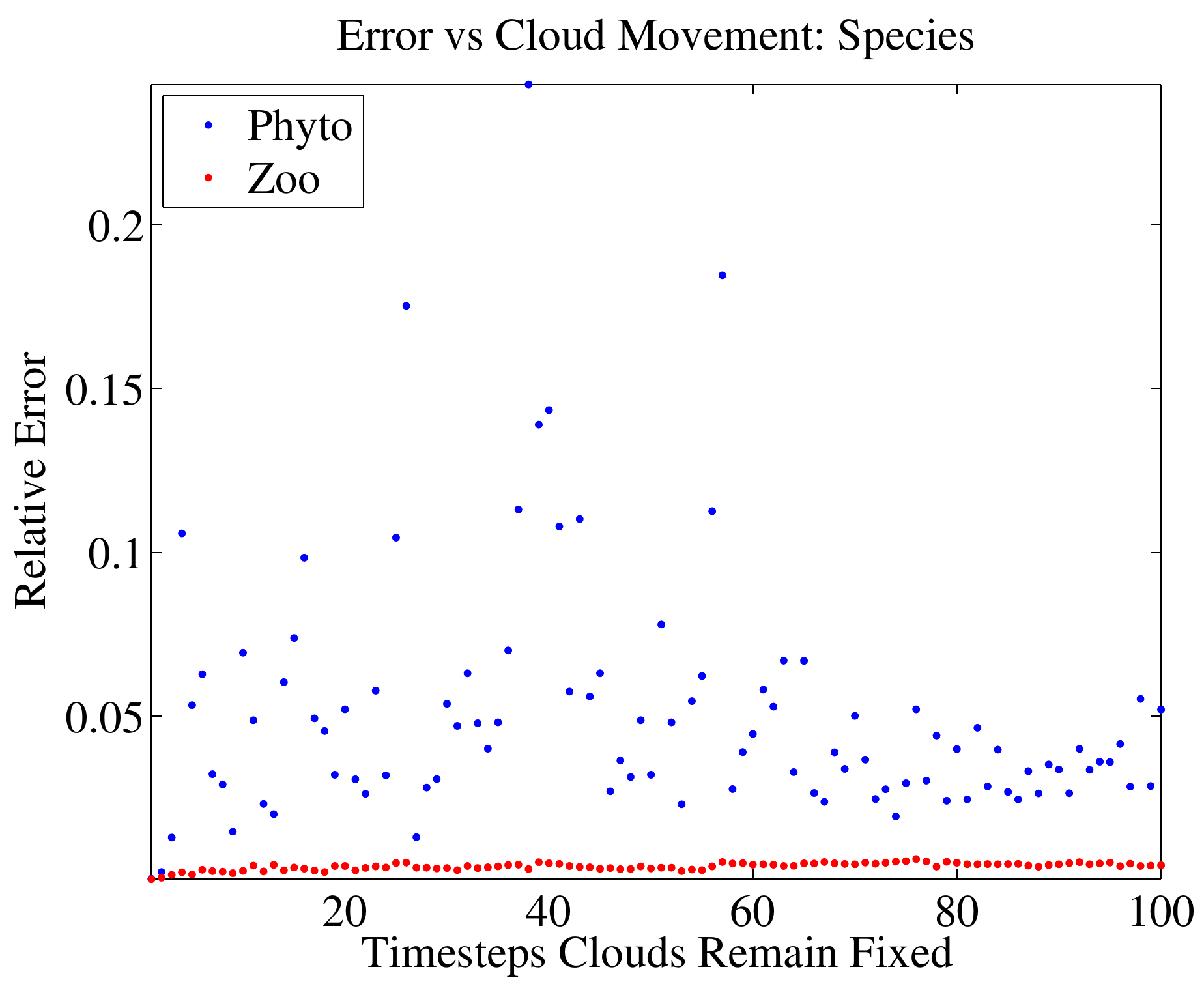}\label{fig:Speed_vs_Error_Species}}  \hspace*{-.1cm}
\subfloat[][]{\includegraphics[width=.24 \textwidth]{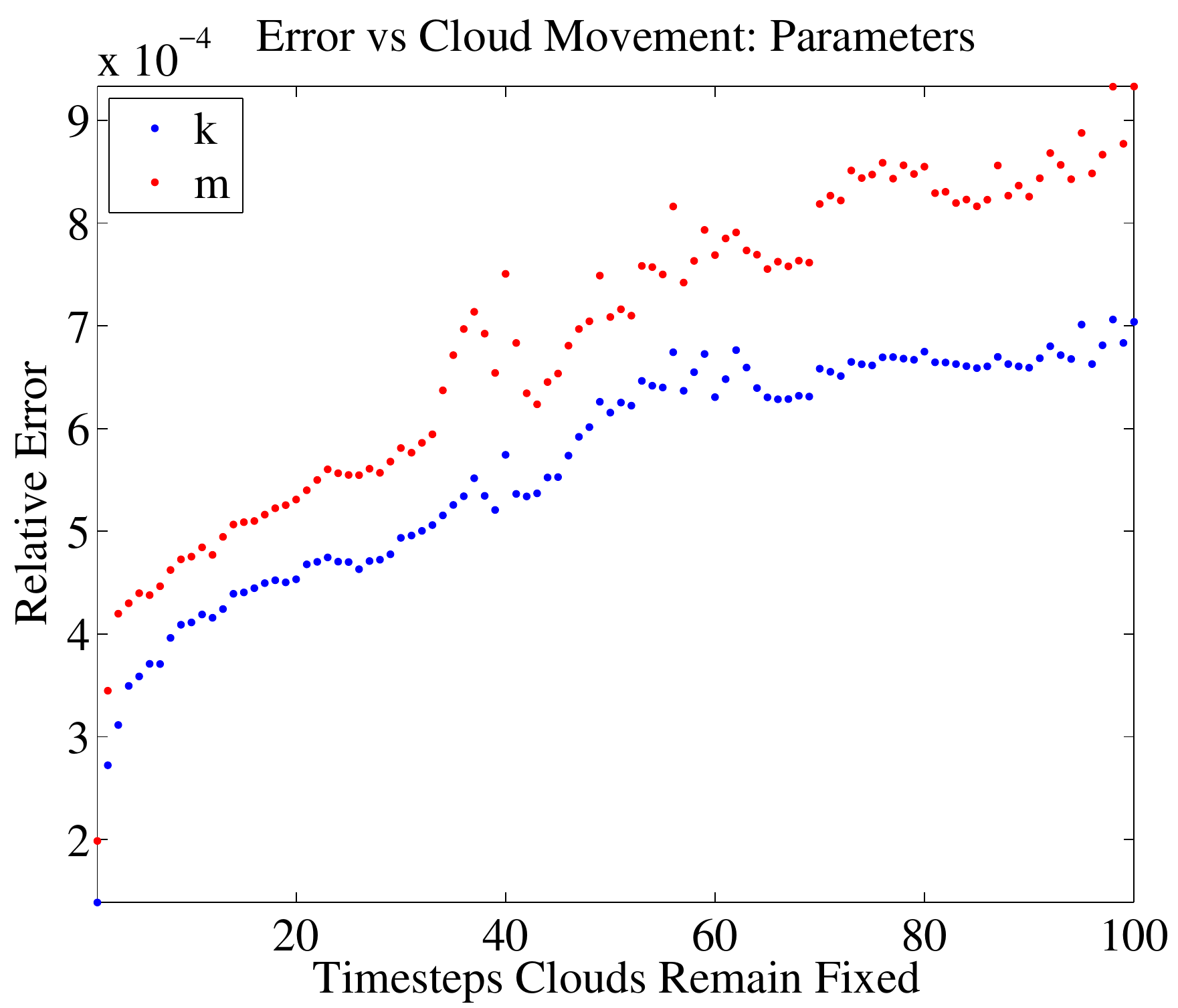}\label{fig:Speed_vs_Error_Parameters}}
\caption{Synchronization error plotted against rate of cloud movement over observations after simulation for $t = 2400$. Species shown in Figure \ref{fig:Speed_vs_Error_Species} and parameters in Figure \ref{fig:Speed_vs_Error_Parameters}.}
\label{fig:errvsspeed}
\end{figure} 

\section{Summary}
We have extended the method of autosynchronization for PDEs with spatially dependent parameters to partially observable noisy PDEs with noisy parameters. We have shown that two PDEs can synchronize together even if the drive system is largely unobservable. Furthermore, all model states and parameters are estimated by sampling only one partially available noisy state. The work is a step toward modeling ocean ecology by remote sensing data over coastal regions or regions with recurring algal blooms. Remote sensing data in the form of hyperspectral satellite imagery often suffer from cloud coverage occluding parts of the domain over which we observe. Parameters and model states are estimated by treating the drive and response systems as independent oscillators on the unobservable set and periodically driving the dynamics toward the synchronization manifold. 

Future work in this area includes adapting the method to work with reaction-diffusion-advection models, and building tools to optimize parameter and state estimation given extremely sparse data. To find a Lyapunov function for the system would provide strong theoretical backing. Furthermore, one might study the observability of the system, or similar systems, in order to understand which state variables are necessary to be sampled for successful estimation \cite{letellier2005relation,letellier2005graphical} 

Algal blooms, especially harmful algal blooms, can have widespread negative consequences on local fisheries and tourism. In effect, models could inform management decisions and provide forecasts for local communities. Coupled with optical flow techniques \cite{Luttman2012AstreamFunctionFrame}, advection could be added to reaction-diffusion models for additional accuracy, particularly over regions for which there are no vector field data for ocean currents. With advection data, techniques such as Finite-time/size-Lyapunov exponents \cite{haller2002lagrangian,bollt2012measurable} or coherent sets based on transfer operator theory \cite{haller98finite,haller2000finding,froyland2009almost,bollt2012measurable} could be used to analyze the coastal dynamics to uncover Lagrangian coherent structures that might inhibit transport between regions of the ocean. Several such techniques could be used in concert to build bloom forecasts for coastal communities, informed by remote sensing data. 

\begin{acknowledgements}
This work has been supported by the Office of Naval Research under N00014-15-1-2093, the Army Research Office under N68164-EG, and a faculty development grant from Norwich University. We would like to thank Dr. Nicholas Tufillaro at Oregon State University's College of Earth, Ocean, and Atmospheric Sciences for providing the satellite image products in the paper. 
\end{acknowledgements}

\bibliography{Clouds_Arxiv.bbl}

\end{document}